\documentclass[11pt]{article}
\usepackage{amsfonts, bm}
\usepackage{amssymb}
\usepackage{latexsym}
\usepackage{mathrsfs}
\usepackage{multirow, lscape}
\usepackage{multicol}
\usepackage{graphicx}
\usepackage{rotating}
\usepackage{subfigure}
\usepackage{color}
\usepackage{algorithm}
\usepackage{algorithmic}
\usepackage{cite}
\usepackage{epstopdf}
\usepackage{amsthm}
\usepackage[cmex10]{amsmath}
\usepackage{bm}
\usepackage{subfigure}
\usepackage{url}
\usepackage{doi}
\usepackage{hyperref}
\usepackage{booktabs}
\usepackage{float}
\usepackage{epstopdf}
\newtheorem{theorem}{Theorem}[section]

\newtheorem{lemma}{Lemma}[section]
\newtheorem{corollary}{Corollary}

\date{}

\begin{document}

\title{A fast second-order implicit difference method for time-space
fractional advection-diffusion equation}

\author{Yong-Liang Zhao$^{a}$ \thanks{\textit{E-mail address:} uestc\_ylzhao@sina.com},
Ting-Zhu Huang$^{a}$ \thanks{Corresponding author. \textit{E-mail address:}
tingzhuhuang@126.com. Tel: +86 28 61831608, fax: +86 28 61831280.},
Xian-Ming Gu$^{b}$ \thanks{\textit{E-mail address:} guxianming@live.cn, x.m.gu@rug.nl},
Wei-Hua Luo$^{c}$ \thanks{\textit{E-mail address:} huaweiluo2012@163.com}\\
{\small{\it a. School of Mathematical Sciences,}}\\
{\small{\it University of Electronic Science and Technology of China,}}\\
{\small{\it Chengdu, Sichuan 611731, P.R. China}}\\
{\small{\it b. School of Economic Mathematics,}}\\
{\small{\it Southwestern University of Finance and Economics,}}\\
{\small{\it Chengdu, Sichuan 611130, P.R. China}}\\
{\small{\it c. Data Recovery Key Laboratory of Sichuan Province,}}\\
{\small{\it Neijiang Normal University, Neijiang, Sichuan, 641100, P.R. China}}\\
}

\maketitle

\numberwithin{equation}{section}
\begin{abstract}
In this paper, we consider a fast and second-order implicit difference method to
approximate a class of linear time-space fractional variable coefficients advection-diffusion equation.
To begin with, we construct an implicit difference scheme based on $L2$-$1_{\sigma}$ formula
[Alikhanov AA. 2015;280:424-38.] for the temporal discretization
and weighted and shifted Gr\"{u}nwald method for the spatial discretization.
Then, unconditional stability of the implicit difference scheme is proved, and
we theoretically and numerically show that it converges in the $L_2$-norm with the optimal order
$\mathcal{O}(\tau^2 + h^2)$ with time step $\tau$ and mesh size $h$.
Moreover, we utilize the same measure to solve the nonlinear case of this problem.
For purpose of effectively solving the discretized systems with the Toeplitz matrix,
two fast Krylov subspace solvers with suitable circulant preconditioners are designed.
In each iterative step, these methods reduce the storage requirements of the resulting
equations from $\mathcal{O}(N^2)$ to $\mathcal{O}(N)$ and the computational complexity from
$\mathcal{O}(N^3)$ to $\mathcal{O}(N \log N)$, where $N$ is the number of grid nodes.
Numerical experiments are carried out to demonstrate that these methods are more practical
than the traditional direct solvers of the implicit difference methods,
in terms of memory requirement and calculation time.



\end{abstract}
\section{Introduction}
\label{sec1}\quad\

This manuscropt is concerned with a fast second-order implicit difference method (IDM)
for solving the initial-boundary value problem of
the time-space fractional advection-diffusion equation (TSFADE)
\cite{liu2007stability,zhao2016preconditioned}:
\begin{equation}
\begin{cases}
\begin{split}
D_{0,t}^{\theta} u(x,t) =
& - d_{+}(t) D_{0,x}^{\alpha} u(x,t) - d_{-}(t) D_{x,L}^{\alpha} u(x,t) \\
& + e_{+}(t) D_{0,x}^{\beta} u(x,t) + e_{-}(t) D_{x,L}^{\beta} u(x,t) + f(x,t),
\end{split}\\
u(x,0) = u_{0}(x), \qquad\qquad\qquad 0 \leq x \leq L,\\
u(0,t) = u(L,t) = 0, \qquad\qquad 0 \leq t \leq T,
\end{cases}
\label{eq1.1}
\end{equation}
where $\theta, \alpha \in (0,1]$, $\beta \in (1,2]$, $0 < x < L$, and $0 < t \leq T$.
Here, the parameters $\theta, \alpha$ and $\beta$ are the order of the TSFADE,
$f(x,t)$ is the source term,
and diffusion coefficient functions $d_{\pm}(t)$ and $e_{\pm}(t)$ are
non-negative under the assumption
that the flow is from left to right.
The TSFADE (\ref{eq1.1}) can be regarded as generalizations of
classical advection-diffusion equations with
the first-order time derivative replaced by the Caputo fractional derivative
of order $\theta \in (0,1]$, and
the first-order and second-order space derivatives replaced
by the two-sided Riemann-Liouville fractional derivatives of order
$\alpha \in (0,1]$ and $\beta \in(1,2]$.
Namely, the time fractional derivative in (\ref{eq1.1}) is
the Caputo fractional derivative\cite{podlubny1998fractional, samko1993fractional} of order $\theta$
denoted by
\begin{equation}
D_{0,t}^{\theta} u(x,t) = \frac{1}{\Gamma(1 - \theta)}
\int_{0}^{t} \frac{\partial u(x,\xi)}{\partial \xi} \frac{d \xi}{(t - \xi)^{\theta}},
\label{eq1.2}
\end{equation}
and the left ($D_{0,x}^{\alpha}$, $D_{0,x}^{\beta}$) and
right ($D_{x,L}^{\alpha}$, $D_{x,L}^{\beta}$) space fractional derivatives in (\ref{eq1.1})
are the Riemann-Liouville fractional derivatives\cite{podlubny1998fractional, samko1993fractional}
 defined as

(1) left Riemann-Liouville fractional derivative:
\begin{equation}
D_{0,x}^{\gamma} u(x,t) = \frac{1}{\Gamma(n - \gamma)} \frac{d^{n}}{dx^{n}}
\int_{0}^{x} \frac{u(\eta,t)}{(x - \eta)^{\gamma - n + 1}} d \eta,
\label{eq1.3}
\end{equation}

(2) right Riemann-Liouville fractional derivative:
\begin{equation}
D_{x,L}^{\gamma} u(x,t) =  \frac{(-1)^{n}}{\Gamma(n - \gamma)} \frac{d^{n}}{dx^{n}}
\int_{x}^{L} \frac{u(\eta,t)}{(\eta - x)^{\gamma - n + 1}} d \eta,
\label{eq1.4}
\end{equation}
where $\Gamma(\cdot)$ denotes the Gamma function and $n - 1 \leq \gamma  < n$ ($n$ is a positive integer).
In reality, when $\theta = \alpha = 1$ and $\beta = 2$,
the above equation (\ref{eq1.1}) reduces to the classical advection-diffusion equation (ADE).

There have been many studies and applications of the fractional ADE.
It should be point out that the fractional ADE provides more adequate and accurate description of
the movement of solute in an
aquifer than traditional second-order ADEs do
\cite{benson2001fractional,benson2000fractional}, and
also used to approach the description of transport dynamics
in complex systems governed by anomalous diffusion
and non-exponential relaxation patterns\cite{metzler2000random}.
Moreover, the fractional ADE is a more suitable model for many problems,
such as engineering, chemistry, entropy\cite{yuriy2015generalization}
and hydrology\cite{benson2000appl}.
To obtain the analytical solutions of fractional partial differential equations
\cite{podlubny1998fractional}, numerous analytical methods,
such as the Fourier transform method,
Adomian¡¯s decomposition method,
the Laplace transform method,
shifted Legendre polynomials\cite{abb2015appl}
and the Mellin transform method,
have been developed.
Since there are very few cases
in which the closed-form analytical solutions are available,
or the obtained analytical solutions are less practical
(expressed by the transcendental functions or infinite series).
Researches on numerical approximation and techniques for
the solution of fractional difference equations (FDEs) have attracted intensive interest;
see \cite{vj2007numerical,vj2001numerical,liu2004numerical,
mm2004finite,mm2006finite,tadjeran2006finite,wang2010direct, luo2016quadratic, gu2015strang}
and references therein.

Regarding numerical methods for fractional advection-diffusion problems.
In the last years,
most early established numerical methods are developed for handling
the space or time fractional ADE.
For the time fractional ADE, many early implicit numerical methods are derived by
the combination of the $L1$ approximate formula\cite{wei2013analysis, zhuang2011time}.
These methods are unconditionally stable, but the time accuracy can not meet second order.
In addition, some other related numerical methods have already been proposed for
handling the time fractional ADE; see
\cite{fu2013boundary, zhai2014uncondition, wang2015compact} for details.
On the other hand, for the space fractional ADE, many researchers exploited the conventional
shifted Gr\"{u}nwald discretization\cite{liu2012numerical} and the implicit Euler (or Crank-Nicolson)
time-step discretization for two-sided Riemann-Liouville fractional derivatives and the first order
time derivative, respectively. Later, Sousa and Li\cite{li2015weighted} derived
a weighted finite difference scheme for producing the novel numerical methods,
which achieved the second order accuracy in both time and space for
space fractional ADE. Qu and Lei et al. employed
circulant and skew-circulant splitting iteration,
which could reduce the required algorithm storage, for implementing the above mentioned second
order numerical method; see \cite{qu2014cscs} for details.

Contrarily, although the numerical methods for space or time
fractional ADE were extensively investigated in the past researches,
the work about numerically handling the TSFADE is not too much.
Firstly, Liu et al. \cite{liu2012numerical} and Zhang \cite{zhang2009afinite,zhang2009finitestfde}
worked out a series of studies about constructing the implicit
difference scheme (IDS) for TSFADE, however all these numerical schemes achieve the
convergence with first order accuracy in both space and time from both the theoretical
and numerical perspectives.
Liu et al. \cite{liu2007stability} also considered a space-time fractional advection dispersion equation
and the implicit proposed difference method has convergence $\mathcal{O}(\tau + h)$.
Then, Shen et al.\cite{shen2011numerical} presented an explicit difference approximation
with convergence $\mathcal{O}(\tau + h)$ and
an implicit difference approximation with convergence rate $\mathcal{O}(\tau^{2 - \alpha} + h)$.
Inspired by Liu's paper \cite{liu2007stability},
Zhao et al.\cite{zhao2016preconditioned} introduced preconditioned iterative methods
to reduce the amount of computation while solving TSFADE,
but its convergence can only reach $\mathcal{O}(\tau^{2 - \alpha}+h)$.
Later, Gu et al. \cite{gu2016fast} established fast solution techniques
to solve the discretization linear system of the time-space fractional convection-diffusion equation,
which has convergence rate $\mathcal{O}(\tau^2 + h^2)$.
Furthermore, most of
these methods have no complete theoretical analysis for both stability and convergence,
refer to \cite{wei2012tsades, hejazi2013finite, irandoust2014numerical, parvizi2015numerical}
for details.

Since the fractional differential operator is nonlocal,
it shows that a naive discretization of the FDE,
even though implicit, leads to unconditionally unstable \cite{mm2004finite,mm2006finite}.
Moreover, traditional methods for solving FDEs tend to
generate full coefficient matrices, which require computational cost of $\mathcal{O}(N^3)$
and storage of $\mathcal{O}(N^2)$ \cite{cui2015compact}. There is no doubt that
if $N$ is not small, both calculation time and storage requirements are tremendous,
and this situation is that we try to avoid. To optimize the computational complexity,
Meerschaet and Tadjeran \cite{mm2004finite,mm2006finite} proposed
a shifted Gr\"{u}nwald discretization to approximate FDEs, which has been proved
to be unconditionally stable. Later, Wang, Wang and Sircar \cite{ww2010direct}
discovered that the full coefficient matrix derived by the shifted Gr\"{u}nwald scheme holds
a Toeplitz-like structure. More precisely, this matrix can be written as a sum of
diagonal-multiply-Toeplitz matrices. This means that
the matrix-vector multiplication for Toeplitz matrix can be computed by
the fast Fourier transform (FFT)
with $\mathcal{O}(N \log N)$ operations \cite{ng2001itertoep}, and the storage requirement
reduces to $\mathcal{O}(N)$. Thanks to the Toeplitz-like structure,
Wang and Wang \cite{ww2011fast} employed the CGNR method to solve discretized linear system
of FDE by the Meerschaet-Tadjeran's method, and the cost of each iteration is
$\mathcal{O}(N \log N)$. Numerical experiments indicate that the CGNR method is fast when
the diffusion coefficients are small enough, namely, the discretized system is well-conditioned.
Nevertheless, the discretized system will become ill-conditioned when the diffusion coefficients
are not small, hence the CGNR method converges very slowly.
To overcome this shortcoming, Zhao et al. extended the precondition technique,
for handling the Toeplitz-like matrix with structure as
the sum of diagonal-multiply-Toeplitz matrices \cite{zhao2016preconditioned}.
Their results are related to the promising acceleration of the
convergence of the iterative methods, while solving (\ref{eq1.1}),
see \cite{gu2015cscs, bai2015hshs} for other preconditioning methods.
For the same reason,
we propose two preconditioned iterative methods,
i.e., the preconditioned biconjugate gradient stabilized
(BiCGSTAB) method \cite{van1992bicgstab}
and the preconditioned generalized product-type solvers
based on BiCOR (GPBiCOR($m$, $\ell$)) method \cite{gu2015bicorstab},
and observe results related to the acceleration of the convergence of
the iterative methods to solve (\ref{eq1.1}).

The rest of this paper is arranged as follows. In Section 2,
the implicit difference scheme for (\ref{eq1.1}) is presented,
and we prove that this scheme is unconditionally stable and convergent
with the accuracy of $\mathcal{O}(\tau^2 + h^2)$.
Then, we extend the problem to nonlinear situation through utilizing the same technique.
In Section 3, the BiCGSTAB method and the GPBiCOR($m$, $\ell$) method with
suitable circulant preconditioners are proposed to solve the discretized system.
In Section 4, numerical results are reported to demonstrate the efficiency of our numerical approaches,
and concluding remarks are given in Section 5.

\section{An implicit difference scheme for TSFADE}
\label{sec2}\quad\

In this section, we present an implicit difference method for discretizing
the TSFADE defined by (\ref{eq1.1}).
Unlike the former numerical approaches \cite{liu2007stability,zhao2016preconditioned,shen2011numerical},
we exploit  henceforth two-sided fractional derivatives to approximate
the Riemann-Liouville derivatives in (\ref{eq1.3}) and (\ref{eq1.4}).
we can show that, by two-sided fractional derivatives,
this proposed method is also unconditionally stable and
convergent second order accuracy in both time and space.

\subsection{Discretization of the TSFADE}
\label{sec2.1}\quad\

In order to derive the proposed scheme,
we first introduce the mesh $\bar{\omega}_{h \tau} = \bar{\omega}_{h} \times \bar{\omega}_{\tau}$,
where
$\bar{\omega}_{h} = \{ x_i =ih, i = 0, 1, \cdots, N; x_0 = 0, x_N = L \}$ and
$\bar{\omega}_{\tau} = \{ t_j = j \tau, j = 0, 1, \cdots, M; t_M = T \}$.
Besides, $\bm{v} = \{ v_{i} \mid 0 \leq i \leq N \} $ be any grid function.
Then the following lemma introduced in \cite{Alikhanov2015424}
gives a description on the Caputo fractional derivative discretization.
\begin{lemma}
suppose $0< \theta <1$, $\sigma = 1 - \frac{\theta}{2}$, $y(t) \in \mathcal{C}^3 [0,T]$, and
$t_{j + \sigma} = (j + \sigma) \tau$.
Then
\begin{equation*}
\left |  D_{0,t}^{\theta} y(t)
- \Delta^{\theta}_{0,t_{j + \sigma}} y(t) \right | = \mathcal{O}(\tau^{3 - \theta}),
\end{equation*}
where
\begin{equation}
\Delta^{\theta}_{0,t_{j + \sigma}} y(t) =
\frac{\tau^{-\theta}}{\Gamma(2 - \theta)} \sum\limits^{j}_{s = 0} c^{(\theta,\sigma)}_{j - s}
\big[ y(t_{s + 1}) - y(t_{s}) \big],
\label{eq2.1}
\end{equation}
and for $j = 0$, $c^{(\theta,\sigma)}_{0} = a^{(\theta,\sigma)}_{0}$, for $j \geq 1$,
\begin{equation*}
c^{(\theta,\sigma)}_{s}=
\begin{cases}
a^{(\theta,\sigma)}_{0} + b^{(\theta,\sigma)}_{1}, & s = 0,\\
a^{(\theta,\sigma)}_{s} + b^{(\theta,\sigma)}_{s + 1} - b^{(\theta,\sigma)}_{s}, & 1\leq s \leq j - 1,\\
a^{(\theta,\sigma)}_{j} - b^{(\theta,\sigma)}_{j}, & s = j,
\end{cases}
\end{equation*}
in which
$a^{(\theta,\sigma)}_{0} = \sigma^{1 - \theta}, ~ a^{(\theta,\sigma)}_{l}
= (l + \sigma)^{1 - \theta} - (l - 1 + \sigma)^{1 - \theta},~ for ~ l \geq 1$;
and
$b^{(\theta,\sigma)}_{l} = \frac{1}{2 - \theta}
\Big[(l + \sigma)^{2 - \theta} - (l - 1 + \sigma)^{2 - \theta} \Big]
- \frac{1}{2} \Big[ (l + \sigma)^{1 - \theta} - (l - 1 + \sigma)^{1 - \theta} \Big]$.
\label{lemma2.1}
\end{lemma}
 On the other hand, the Riemann-Liouville derivatives (\ref{eq1.3}) and (\ref{eq1.4})
can be approximated by the weighted and shifted Gr\"{u}nwald difference (WSGD) operator
in Deng's paper\cite{Deng20150606} (in this paper $(p,q) = (1,0)$),
for parameter $\gamma$, i.e.,
\begin{equation}
D_{0,x}^{\gamma} u(x,t) = \frac{1}{h^{\gamma}}
\sum\limits_{k = 0}^{\left [ \frac{x}{h} \right ]+1}
\omega_{k}^{(\gamma)} u(x - (k - 1) h,t) + \mathcal{O}(h^2),
\label{eq2.2}
\end{equation}
\begin{equation}
D_{x,L}^{\gamma} u(x,t) = \frac{1}{h^{\gamma}}
\sum\limits_{k = 0}^{\left [ \frac{L - x}{h} \right ]+1}
\omega_{k}^{(\gamma)} u(x + (k - 1) h,t) + \mathcal{O}(h^2).
\label{eq2.3}
\end{equation}
Let $u^{j}_{i}$ represent the numerical approximation of $u(x_i,t_j)$, then we denote that
\begin{equation*}
\delta^{\gamma}_{x,+} u_{i}^{j} = \frac{1}{h^{\gamma}}
\sum\limits_{k = 0}^{i +1}
\omega_{k}^{(\gamma)} u_{i - k + 1}^{j}, \quad
\delta^{\gamma}_{x,-} u_{i}^{j} = \frac{1}{h^{\gamma}}
\sum\limits_{k = 0}^{N - i +1}
\omega_{k}^{(\gamma)} u_{i + k - 1}^{j},
\end{equation*}
where
\begin{equation*}
\omega_{0}^{(\gamma)} = \frac{\gamma}{2} g_{0}^{(\gamma)}, \qquad
\omega_{k}^{(\gamma)} = \frac{\gamma}{2} g_{k}^{(\gamma)}
+ \frac{2 - \gamma}{2} g_{k - 1}^{(\gamma)}, ~ k \geq 1,
\end{equation*}
and
\begin{equation*}
g_{0}^{(\gamma)} = 1, \qquad g_{k}^{(\gamma)} = \Big( 1 - \frac{\gamma + 1}{k} \Big)
g_{k - 1}^{(\gamma)}, ~ k = 1, 2, \ldots
\end{equation*}

Suppose $u(x,t) \in \mathcal{C}^{4,3}_{x,t} [0,L] \times [0,T]$
is a solution of the TSADE (\ref{eq1.1}).
For simplicity, we introduce some symbols
\begin{equation*}
u_{i}^{(j + \sigma)} = \sigma u_{i}^{j+1} + (1 - \sigma) u_{i}^{j}, \quad
d_{\pm}^{j + \sigma} = d_{\pm}(t_{j + \sigma}), \quad
e_{\pm}^{j + \sigma} = e_{\pm}(t_{j + \sigma}), \quad
f_{i}^{j + \sigma} = f(x_i,t_{j + \sigma}),
\end{equation*}
\begin{equation*}
\delta^{\alpha, \beta}_{h} u_{i}^{(j + \sigma)} =
- d_{+}^{j + \sigma} \delta^{\alpha}_{x,+} u_{i}^{(j + \sigma)}
- d_{-}^{j + \sigma} \delta^{\alpha}_{x,-} u_{i}^{(j + \sigma)}
 + e_{+}^{j + \sigma} \delta^{\beta}_{x,+} u_{i}^{(j + \sigma)}
+ e_{-}^{j + \sigma} \delta^{\beta}_{x,-} u_{i}^{(j + \sigma)}.
\end{equation*}

Using (\ref{eq2.1})-(\ref{eq2.3}), we shall see that the solution of (\ref{eq1.1})
can be approximated by the following implicit difference scheme
for $(x,t) = (x_i,t_{j + \sigma}) \in \bar{\omega}_{h \tau}$, $i = 1, 2, \ldots, N-1$,
$j = 0, 1, \ldots, M-1$:
\begin{equation*}
\Delta^{\theta}_{0,t_{j + \sigma}} u_i =
\delta^{\alpha, \beta}_{h} u_{i}^{(j + \sigma)} + f_{i}^{j + \sigma}.
\end{equation*}

Then we derive the implicit difference scheme with the approximation order $\mathcal{O}(\tau^2 + h^2)$:
\begin{equation}
\begin{cases}
\Delta^{\theta}_{0,t_{j + \sigma}} u_i =
\delta^{\alpha, \beta}_{h} u_{i}^{(j + \sigma)} + f_{i}^{j + \sigma},
& 1 \leq i \leq N-1,~~ 0 \leq j \leq M-1, \\
u_{i}^{0} = u_{0}(x_i), & 1 \leq i \leq N-1, \\
u_{0}^{j} = u_{N}^{j} = 0, & 0 \leq j \leq M-1.
\end{cases}
\label{eq2.4}
\end{equation}

It is interesting to note that for $\theta \rightarrow 1$, we obtain the Crank-Nicolson difference scheme.

\subsection{Analysis of the implicit difference scheme}
\label{sec2.2}\quad\

In this subsection, we analyze the stability and convergence for the implicit difference scheme (\ref{eq2.4}).
Before proving the stability and convergence of the implicit difference scheme (\ref{eq2.4}),
we provide some properties of the coefficients $\omega_{k}^{(\alpha)}$ and $\omega_{k}^{(\beta)}$,
which are derived from $g_{k}^{(\alpha)}$ and $g_{k}^{(\beta)}$
\cite{liu2007stability,feng2016high,zhao2016preconditioned}, respectively.
\begin{lemma}\cite{feng2016high}
Suppose that $0< \alpha <1$, then the coefficients $\omega_{k}^{(\alpha)}$ satisfy
\begin{equation*}
\begin{cases}
\omega_{0}^{(\alpha)} = \frac{\alpha}{2} > 0, ~~
\omega_{1}^{(\alpha)} = \frac{2 - \alpha - \alpha^2}{2} > 0, ~~
\omega_{2}^{(\alpha)} = \frac{\alpha (\alpha^2 + \alpha - 4)}{4} < 0, \\
\omega_{2}^{(\alpha)} < \omega_{3}^{(\alpha)} < \omega_{4}^{(\alpha)} < \cdots < 0, ~~
\omega_{0}^{(\alpha)} + \omega_{2}^{(\alpha)} < 0, \\
\sum\limits_{k = 0}^{\infty} \omega_{k}^{(\alpha)} = 0, ~~
\sum\limits_{k = 0}^{N} \omega_{k}^{(\alpha)} > 0, ~~N \geq 1. \\
\end{cases}
\end{equation*}
\label{lemma2.2}
\end{lemma}
\begin{lemma}\cite{feng2016high}
Suppose that $1< \beta <2$, then the coefficients $\omega_{k}^{(\beta)}$ satisfy
\begin{equation*}
\begin{cases}
\omega_{0}^{(\beta)} = \frac{\beta}{2} > 0, ~~
\omega_{1}^{(\beta)} = \frac{2 - \beta - \beta^2}{2} < 0, ~~
\omega_{2}^{(\beta)} = \frac{\beta (\beta^2 + \beta - 4)}{4}, \\
1 \geq \omega_{0}^{(\beta)} \geq \omega_{3}^{(\beta)} \geq \omega_{4}^{(\beta)} \geq \cdots \geq 0, ~~
\omega_{0}^{(\beta)} + \omega_{2}^{(\beta)} > 0, \\
\sum\limits_{k = 0}^{\infty} \omega_{k}^{(\beta)} = 0, ~~
\sum\limits_{k = 0}^{N} \omega_{k}^{(\beta)} < 0, ~~N \geq 2.
\end{cases}
\end{equation*}
\label{lemma2.3}
\end{lemma}

In the rest of this paper,
we define the discrete inner product and the corresponding discrete $L_2$-norm as follows,
\begin{equation*}
(\bm{u}, \bm{v}) = h \sum\limits_{i = 1}^{N - 1} u_i v_i,
\quad \left \| \bm{u} \right \| = \sqrt{(\bm{u}, \bm{u})}, \quad for~\forall \bm{u}, ~\bm{v} \in V_h,
\end{equation*}
where
\begin{equation*}
V_h = \{ \bm{v} \mid \bm{v} =
\{ v_i \} ~ is ~ a ~ grid ~ function ~ on ~ \bar{\omega}_{h} ~ and ~ v_i = 0 ~ if ~ i = 0,N \}.
\end{equation*}

Based on the above lemmas, we will provide some other lemmas and corollaries,
which are important properties to proof the stability and the convergence
of the implicit difference scheme (\ref{eq2.4}).
\begin{lemma}
For $0< \alpha <1$, and any $\bm{v} \in V_h$, it holds that
\begin{equation*}
(\delta^{\alpha}_{x,+} \bm{v}, \bm{v}) = (\delta^{\alpha}_{x,-} \bm{v}, \bm{v}) \geq
\Big( \frac{1}{h^\alpha} \sum\limits_{k = 0}^{N - 1} \omega_{k}^{(\alpha)} \Big) \left \| \bm{v} \right \|^2.
\end{equation*}
\label{lemma2.4}
\end{lemma}
\begin{proof}
According to the discrete definitions of $\delta^{\alpha}_{x,+}$ and $\delta^{\alpha}_{x,-}$, and
$v_0 = v_N = 0$, we have
\begin{equation*}
(\delta^{\alpha}_{x,+} \bm{v}, \bm{v}) = (\delta^{\alpha}_{x,-} \bm{v}, \bm{v})
\end{equation*}
and
\begin{equation*}
\begin{split}
(\delta^{\alpha}_{x,+} \bm{v}, \bm{v}) & = h^{1 - \alpha} \sum\limits_{i = 1}^{N - 1} \Big(
\sum\limits_{k = 0}^{i + 1} \omega_{k}^{(\alpha)} v_{i - k + 1} \Big) v_i \\
& = \omega_{1}^{(\alpha)} h^{1 - \alpha} \sum\limits_{i = 1}^{N - 1} (v_i)^2
+ (\omega_{0}^{(\alpha)} + \omega_{2}^{(\alpha)}) h^{1 - \alpha} \sum\limits_{i = 1}^{N - 2} v_{i} v_{i + 1}
+ \sum\limits_{k = 3}^{N - 1} \omega_{k}^{(\alpha)} h^{1 - \alpha} \sum\limits_{i = 1}^{N - k} v_{i + k - 1} v_{i} \\
& \geq \omega_{1}^{(\alpha)} h^{ - \alpha} \left \| \bm{v} \right \|^2
+ (\omega_{0}^{(\alpha)} + \omega_{2}^{(\alpha)}) h^{1 - \alpha}
\sum\limits_{i = 1}^{N - 2} \frac{(v_{i})^2 + (v_{i + 1})^2}{2} \\
&~~~+ \sum\limits_{k = 3}^{N - 1} \omega_{k}^{(\alpha)} h^{1 - \alpha}
\sum\limits_{i = 1}^{N - k} \frac{(v_{i + k - 1})^2 + (v_{i})^2}{2}
 \geq \Big( \frac{1}{h^\alpha} \sum\limits_{k = 0}^{N - 1} \omega_{k}^{(\alpha)} \Big) \left \| \bm{v} \right \|^2.
\end{split}
\end{equation*}
\end{proof}
Further, the following conclusion is true.
\begin{corollary}
For $0< \alpha <1$, and any $\bm{v} \in V_h, ~ N \geq 2$, there exists a positive constant $c_1$, such that
\begin{equation*}
(\delta^{\alpha}_{x,+} \bm{v}, \bm{v}) = (\delta^{\alpha}_{x,-} \bm{v}, \bm{v})
> c_1 \ln 2 \left \| \bm{v} \right \|^2.
\end{equation*}
\label{coro1}
\end{corollary}
\begin{proof}
Since
\begin{equation*}
\sum\limits_{k = N}^{2N + 2} \omega_{k}^{(\alpha)} =
\sum\limits_{k = N}^{2N} g_{k}^{(\alpha)} + \frac{2 - \alpha}{2} g_{N - 1}^{(\alpha)}
+ \frac{\alpha}{2} g_{2N + 2}^{(\alpha)} + g_{2N + 1}^{(\alpha)}.
\end{equation*}

According to the nature of $g_{k}^{(\alpha)}$ \cite{liu2007stability,feng2016high,zhao2016preconditioned},
the following inequality is established:

\begin{equation*}
\sum\limits_{k = N}^{2N + 2} \omega_{k}^{(\alpha)} < \sum\limits_{k = N}^{2N} g_{k}^{(\alpha)},
\qquad N \geq 2.
\end{equation*}

Then, there exist two positive constants $\tilde{c}_1$ and $c_1$,
such that
\begin{equation*}
\begin{split}
\frac{1}{h^\alpha} \sum\limits_{k = N}^{\infty} \Big( -\omega_{k}^{(\alpha)} \Big)
& > \frac{1}{h^\alpha} \sum\limits_{k = N}^{2N + 2} \Big( -\omega_{k}^{(\alpha)} \Big)
> \frac{1}{h^\alpha} \sum\limits_{k = N}^{2N} \Big( -g_{k}^{(\alpha)} \Big) \\
& \geq \tilde{c}_1 \sum\limits_{k = N}^{2N} k^{-(\alpha + 1)} N^\alpha
>  \tilde{c}_1 \sum\limits_{k = N}^{2N} k^{-(\alpha + 1)} \Big( \frac{k}{2} \Big)^\alpha \\
& = \frac{\tilde{c}_1}{2^\alpha} \sum\limits_{k = N}^{2N} \frac{1}{k}
> c_1 \int\limits_{N}^{2N + 1} \frac{1}{x} dx \geq c_1 \int\limits_{N}^{2N} \frac{1}{x} dx
= c_1 \ln2, \quad N \geq 2.
\end{split}
\end{equation*}

Here, the penultimate and antepenult inequalities are true due to the fact that $J_1(x) = 1/x$
is a lower convex function and $J_2(x) = \ln(x) > 0, ~ x \in [N, 2N + 1]$ is an increasing function.

Combining with Lemmas \ref{lemma2.2} and \ref{lemma2.4}, we obtain

\begin{equation*}
(\delta^{\alpha}_{x,+} \bm{v}, \bm{v}) = (\delta^{\alpha}_{x,-} \bm{v}, \bm{v}) \geq
\frac{1}{h^\alpha} \sum\limits_{k = N}^{\infty} \Big( -\omega_{k}^{(\alpha)} \Big) \left \| \bm{v} \right \|^2
> c_1 \ln 2 \left \| \bm{v} \right \|^2.
\end{equation*}
\end{proof}
Using the same argument as in the proof of Corollary \ref{coro1},
we can easily carry out the following corollary.
\begin{corollary}
(1) For $1< \beta<2$, and any $\bm{v} \in V_h$, it holds that
\begin{equation*}
(\delta^{\beta}_{x,+} \bm{v}, \bm{v}) = (\delta^{\beta}_{x,-} \bm{v}, \bm{v}) \leq
\Big( \frac{1}{h^\beta} \sum\limits_{k = 0}^{N - 1} \omega_{k}^{(\beta)} \Big) \left \| \bm{v} \right \|^2.
\end{equation*}

(2) For $1< \beta<2$, and any $\bm{v} \in V_h, ~ N \geq 3$, there exists a positive constant $c_2$, such that
\begin{equation*}
(-\delta^{\beta}_{x,+} \bm{v}, \bm{v}) = (-\delta^{\beta}_{x,-} \bm{v}, \bm{v})
> c_2 \ln 2 \left \| \bm{v} \right \|^2.
\end{equation*}
\label{coro2}
\end{corollary}
With the help of the preceding lemmas and corollaries, we can now establish the following theorem,
which is essential for analyzing the stability of the proposed implicit difference scheme.
\begin{theorem}
For any $\bm{v} \in V_h$, it holds that
\begin{equation*}
(\delta^{\alpha, \beta}_{h} \bm{v}, \bm{v}) \leq -c \ln 2 \left \| \bm{v} \right \|^2,
\end{equation*}
where $c$ is a positive constant independent of the spatial step size $h$.
\label{th2.1}
\end{theorem}
\begin{proof}
The concrete expression of $(\delta^{\alpha, \beta}_{h} \bm{v}, \bm{v})$
can be written as
\begin{equation*}
(\delta^{\alpha, \beta}_{h} \bm{v}, \bm{v}) =
- d_{+}^{j + \sigma} (\delta^{\alpha}_{x,+} \bm{v}, \bm{v})
- d_{-}^{j + \sigma} (\delta^{\alpha}_{x,-} \bm{v}, \bm{v})
 + e_{+}^{j + \sigma} (\delta^{\beta}_{x,+} \bm{v}, \bm{v})
+ e_{-}^{j + \sigma} (\delta^{\beta}_{x,-} \bm{v}, \bm{v}).
\end{equation*}

According to Corollaries \ref{coro1}-\ref{coro2}, there exist two positive constants $c_1$ and $c_2$ independent of the
spatial step size $h$, such that for any non-vanishing vector $\bm{v} \in V_h$, we have
\begin{equation*}
\begin{split}
(\delta^{\alpha, \beta}_{h} \bm{v}, \bm{v}) & =
- d_{+}^{j + \sigma} (\delta^{\alpha}_{x,+} \bm{v}, \bm{v})
- d_{-}^{j + \sigma} (\delta^{\alpha}_{x,-} \bm{v}, \bm{v})
+ e_{+}^{j + \sigma} (\delta^{\beta}_{x,+} \bm{v}, \bm{v})
+ e_{-}^{j + \sigma} (\delta^{\beta}_{x,-} \bm{v}, \bm{v}) \\
& < -\ln 2 \Big[ c_1 \big( d_{+}^{j + \sigma} + d_{-}^{j + \sigma} \big)
+ c_2 \big( e_{+}^{j + \sigma} + e_{-}^{j + \sigma} \big) \Big] \left \| \bm{v} \right \|^2.
\end{split}
\end{equation*}

For simplicity, we may take $c = c_1 \big( d_{+}^{j + \sigma} + d_{-}^{j + \sigma} \big)
+ c_2 \big( e_{+}^{j + \sigma} + e_{-}^{j + \sigma} \big)$. Hence, the targeted result is
immediately completed.
\end{proof}

Now we can proof the stability and convergence of the implicit difference scheme (\ref{eq2.4}).
For convenience, in our proof, we denote
$q_{s}^{j + 1} = \frac{\tau^{-\theta} c_{j - s}^{(\theta, \sigma)}}{\Gamma(2 - \theta)}$.
Then $\Delta_{0,t_{j + \sigma}}^{\theta} = \sum_{s = 0}^{j} \big( u^{s + 1} - u^s \big) q_{s}^{j + 1}$.
\begin{theorem}
Denote $u^{j + 1} = (u_1^{j + 1}, u_2^{j + 1}, \ldots, u_{N - 1}^{j + 1})^T$ and
$\left \| f^{j + \sigma} \right \|^2 = h \sum\limits_{i = 1}^{N - 1} f^2 (x_i, t_{j + \sigma})$.
Then the implicit difference scheme (\ref{eq2.4})
is unconditionally stable for $N \geq 3$, and the following a priori estimate holds:
\begin{equation}
\left \| u^{j + 1} \right \|^2 \leq \left \| u^{0} \right \|^2
+ \frac{T^{\theta} \Gamma(1 - \theta)}{c \ln 2} \left \| f^{j + \sigma} \right \|^2,
\quad 0 \leq j \leq M - 1.
\label{eq2.5}
\end{equation}
\label{th2.2}
\end{theorem}
\begin{proof}
Taking the inner product of (\ref{eq2.4}) with $u^{(j + \sigma)}$, we have
\begin{equation*}
(\Delta_{0,t_{j + \sigma}}^{\theta} u, u^{(j + \sigma)}) =
(\delta^{\alpha, \beta}_{h} u^{(j + \sigma)}, u^{(j + \sigma)}) + (f^{j + \sigma}, u^{(j + \sigma)}).
\end{equation*}

Using Corollary 1 in \cite{Alikhanov2015424} and Theorem \ref{th2.1}, obtains

\begin{equation}
\frac{1}{2} \Delta_{0,t_{j + \sigma}}^{\theta} \left \| u \right \|^2
\leq -c \ln 2 \left \| u^{(j + \sigma)} \right \|^2
+ \varepsilon \left \| u^{(j + \sigma)} \right \|^2 + \frac{1}{4 \varepsilon} \left \| f^{j + \sigma} \right \|^2,
\quad \varepsilon > 0.
\label{eq2.6}
\end{equation}

From (\ref{eq2.6}), at $\varepsilon = c \ln 2$ we get
\begin{equation}
\frac{1}{2} \Delta_{0,t_{j + \sigma}}^{\theta} \left \| u \right \|^2
\leq \frac{1}{4 c \ln 2} \left \| f^{j + \sigma} \right \|^2.
\label{eq2.7}
\end{equation}

Let us rewrite inequality (\ref{eq2.7}) in the form
%
\begin{equation}
q_j^{j + 1} \left \| u^{j + 1} \right \|^2
\leq
\sum_{s = 1}^{j} (q_s^{j+ 1} - q_{s - 1}^{j + 1}) \left \| u^{s} \right \|^2
+ q_0^{j + 1} \Big( \left \| u^{0} \right \|^2
+ \frac{T^{\theta} \Gamma(1 - \theta)}{c \ln 2} \left \| f^{j + \sigma} \right \|^2 \Big),
\label{eq2.8}
\end{equation}
noticing that $q_0^{j + 1} > \frac{1}{2 T^{\theta} \Gamma(1 - \theta)}$ (cf.\cite{dehghan2015two}).

With the inequality (\ref{eq2.8}) in hand, the inequality (\ref{eq2.5}) can be reached through utilizing
the mathematical induction method.

The proof of Theorem \ref{th2.2} is completed.
\end{proof}
With the above proof, the convergence of the implicit difference scheme (\ref{eq2.4})
is easy to obtain.
\begin{theorem}
Suppose that $u(x,t)$ is the solution of (\ref{eq1.1})
and $\{ u_i^j \mid x_i \in \bar{\omega}_{h}, ~~ 0 \leq j \leq M \}$
is the solution of the implicit difference scheme (\ref{eq2.4}).
Denote
\begin{equation*}
\xi_i^j = u(x_i,t_j) - u_i^j, ~~ x_i \in \bar{\omega}_{h}, ~~ 0 \leq j \leq M.
\end{equation*}

Then there exists a positive constant $\tilde{c}$ such that
\begin{equation*}
\left \| \xi^j \right \| \leq \tilde{c} (\tau^2 + h^2), ~~ 0 \leq j \leq M.
\end{equation*}
\label{th2.3}
\end{theorem}
\begin{proof}
It can easily obtain that $\xi^j$ satisfies the following error equation
\begin{equation*}
\begin{cases}
\Delta_{0,t_{j + \sigma}}^{\theta} \xi_i = \delta_{h}^{\alpha, \theta} \xi_{i}^{(j + \sigma)}
+ R_{i}^{j + \sigma}, & 1 \leq i \leq N - 1, ~~ 0 \leq j \leq M - 1,\\
\xi_i^0 = 0, & 1 \leq i \leq N - 1, \\
\xi_0^j = \xi_N^j = 0, & 0 \leq j \leq M,
\end{cases}
\end{equation*}

where $R_{i}^{j + \sigma} = \mathcal{O} (\tau^2 + h^2)$.
By exploiting Theorem \ref{th2.2} we get
\begin{equation*}
\left \| \xi^{j + 1} \right \|^2
\leq  \frac{T^{\theta} \Gamma(1 - \theta)}{c \ln2} \left \| R^{j + \sigma} \right \|^2
\leq \tilde{c} (\tau^2 + h^2), ~~ 0 \leq j \leq M - 1,
\end{equation*}
which implies the convergence in the mesh $L_2$-norm with rate $\mathcal{O} (\tau^2 + h^2)$.
\end{proof}
The unconditionally stability and convergence of IDS (\ref{eq2.4}) have been proved,
and a numerical test is displayed to verify our results in Section 4.
\subsection{A nonlinear extension of TSFADE}
\label{sec2.3} \quad\
In this subsection, we extend the TSFADE to nonlinear case:
\begin{equation}
\begin{cases}
\begin{split}
D_{0,t}^{\theta} u(x,t) =
& - d_{+}(t) D_{0,x}^{\alpha} u(x,t) - d_{-}(t) D_{x,L}^{\alpha} u(x,t) \\
& + e_{+}(t) D_{0,x}^{\beta} u(x,t) + e_{-}(t) D_{x,L}^{\beta} u(x,t) + g(u, x, t),
\end{split}\\
u(x,0) = u_{0}(x), \qquad\qquad\qquad 0 \leq x \leq L,\\
u(0,t) = u(L,t) = 0, \qquad\qquad 0 \leq t \leq T,
\end{cases}
\label{eq2.9}
\end{equation}
where $g(u(x, t),x,t) = f(x,t) + y(u(x, t))$.

Considering (\ref{eq2.9}) at point $(x_i, t_{j + \sigma})$ and doing some simple manipulations,
we have
\begin{equation}
\begin{cases}
\Delta^{\theta}_{0,t_{j + \sigma}} u_i =
\delta^{\alpha, \beta}_{h} u_{i}^{(\sigma)} + g_{i}^{j + \sigma},
& 1 \leq i \leq N-1,~~ 0 \leq j \leq M-1, \\
u_{i}^{0} = u_{0}(x_i), & 1 \leq i \leq N-1, \\
u_{0}^{j} = u_{N}^{j} = 0, & 0 \leq j \leq M-1,
\end{cases}
\label{eq2.10}
\end{equation}
in which $g_{i}^{j + \sigma} = g(\sigma u_{i}^{j + 1} + (1 - \sigma) u_{i}^{j}, x_i, t_{ j + \sigma})$,
and the truncation error is $\mathcal{O}(\tau^2 + h^2)$.

We assume $U_{i}^{j}$ is the approximation solution of $u_{i}^{j}$ given in (\ref{eq2.10}),
and let $\xi_{i}^{j} = u_{i}^{j} - U_{i}^{j},~i = 1, \cdots, N - 1;~j = 1, \cdots, M - 1$,
be the error satisfying
\begin{equation}
\Delta^{\theta}_{0,t_{j + \sigma}} \zeta_i =
\delta^{\alpha, \beta}_{h} \zeta_{i}^{(\sigma)} + y(u_{i}^{j + \sigma}) - y(U_{i}^{j + \sigma}).
\label{eq2.11}
\end{equation}
Correspondingly, we denote $E^{j} = \left[ \zeta_{1}^{j}, \zeta_{2}^{j}, \cdots, \zeta_{N - 1}^{j} \right]^{T}
~(j = 1, \cdots, M - 1)$, and suppose the nonlinear term $y(u(x, t))$ satisfies:
\begin{equation*}
 | y(u(x,t)) - y(v(x,t)) | \leq L_1 | u(x,t) - v(x,t) |, ~for~all~u(x,t), v(x,t) ~over~[0, L] \times [0, T]
 \end{equation*}
 with
$0 < L_1 < \frac{a_{0}^{(\theta,\sigma)}}{2 \left[ \sigma^2 + (1 - \sigma)^2 \right]
\tau^2 \Gamma(2 - \theta)}$.
Now, we proof the stability of (\ref{eq2.10}).
\begin{theorem}
The implicit difference scheme  (\ref{eq2.10}) is stable, i.e.:
\begin{equation*}
\left \| E^{j + 1} \right \| \leq \tilde{C} \left \| E^{0} \right \|, ~j = 0,1,\cdots, M - 1,
\end{equation*}
where $\tilde{C}$ is a positive constant, which may independent of $h$.
\label{th2.4}
\end{theorem}
\begin{proof}
Taking the inner product of (\ref{eq2.11}) with $E^{(j + \sigma)}$, we have
\begin{equation*}
(\Delta_{0,t_{j + \sigma}}^{\theta} E, E^{(j + \sigma)}) =
(\delta^{\alpha, \beta}_{h} E^{(j + \sigma)}, E^{(j + \sigma)})
+ (y(u^{j + \sigma}) - y(U^{j + \sigma}), E^{(j + \sigma)}).
\end{equation*}
Utilizing Corollary 1 in \cite{Alikhanov2015424}
and noticing $(\delta^{\alpha, \beta}_{h} E^{(j + \sigma)}, E^{(j + \sigma)}) < 0$, gets
\begin{equation}
\frac{1}{2} \Delta_{0,t_{j + \sigma}}^{\theta} \left \| E \right \|^2
\leq L_1\left \| E^{(j + \sigma)} \right \|^2.
\label{eq2.12}
\end{equation}
The right hand side can be estimated as below
\begin{equation*}
\begin{split}
\left \| E^{(j + \sigma)} \right \|^2 & =
\sigma^2 \left \| E^{j + 1} \right \|^2 + (1 - \sigma)^2 \left \| E^{j} \right \|^2
+ 2 \left( \sigma E^{j + 1}, (1 - \sigma) E^{j} \right) \\
& = \sigma^2 \left \| E^{j + 1} \right \|^2 + (1 - \sigma)^2 \left \| E^{j} \right \|^2
+ 2 \left( (1 - \sigma) E^{j + 1}, \sigma E^{j} \right) \\
& \leq \left[ \sigma^2 + (1 - \sigma)^2 \right] \left( \left \| E^{j + 1} \right \|^2 + \left \| E^{j } \right \|^2 \right).
\end{split}
\end{equation*}
Bringing the above estimate to inequality (\ref{eq2.12}), obtains
\begin{equation*}
\begin{split}
& \left\{ \frac{c_{0}^{(\theta, \sigma)}}{2 \tau^{\theta} \Gamma(2 - \theta)}
- L_1 \left[ \sigma^2 + (1 - \sigma)^2 \right] \right\} \left \| E^{j + 1} \right \|^2 \\
&\quad \leq
L_1 \left[ \sigma^2 + (1 - \sigma)^2 \right] \left \| E^{j} \right \|^2
+ \frac{c_{j}^{(\theta, \sigma)}}{2 \tau^{\theta} \Gamma(2 - \theta)} \left \| E^{0} \right \|^2 \\
&\qquad\quad + \frac{1}{2 \tau^{\theta} \Gamma(2 - \theta)} \sum_{s = 1}^{j} \left( c_{j - s}^{(\theta, \sigma)}
- c_{j - s + 1}^{(\theta, \sigma)} \right) \left \| E^{s} \right \|^2.
\end{split}
\end{equation*}
Employing the mathematical induction method, the targeted result is immediately completed.
\end{proof}
Next, we study the convergence of (\ref{eq2.10}).
\begin{theorem}
Suppose that $u(x,t)$ is the solution of (\ref{eq2.9})
and $\{ u_i^j \mid x_i \in \bar{\omega}_{h}, ~~ 0 \leq j \leq M \}$
is the solution of the implicit difference scheme (\ref{eq2.10}).
Denote
\begin{equation*}
\psi_i^j = u(x_i,t_j) - u_i^j, ~~ x_i \in \bar{\omega}_{h}, ~~ 0 \leq j \leq M.
\end{equation*}
Then there exists a positive constant $\hat{C}$ such that
\begin{equation*}
\left \| \psi^j \right \| \leq \hat{C} (\tau^2 + h^2), ~~ 0 \leq j \leq M.
\end{equation*}
\label{th2.5}
\end{theorem}
\begin{proof}
Subtract (\ref{eq2.10}) from (\ref{eq2.9}), obtains
\begin{equation*}
\begin{cases}
\Delta^{\theta}_{0,t_{j + \sigma}} \psi_i =
\delta^{\alpha, \beta}_{h} \psi_{i}^{(\sigma)} + y(u(x_i, t_{j + \sigma})) - y(u_{i}^{j + \sigma}) + \tilde{R}_{i}^{j + \sigma},
& 1 \leq i \leq N-1,~~ 0 \leq j \leq M-1, \\
\psi_{i}^{0} = 0, & 1 \leq i \leq N-1, \\
\psi_{0}^{j} = \psi_{N}^{j} = 0, & 0 \leq j \leq M-1,
\end{cases}
\end{equation*}
where $\tilde{R}_{i}^{j + \sigma} = \mathcal{O}(\tau^2 + h^2)$.
For convenience, we assume that there is a positive constant $\bar{C}$ such that:
$\tilde{R}_{i}^{j + \sigma} \leq \bar{C} (\tau^2 + h^2),~1 \leq i \leq N-1,~ 0 \leq j \leq M-1$.

Taking inner product of this error equation with $\psi^{j + \sigma}$, we have
\begin{equation*}
(\Delta_{0,t_{j + \sigma}}^{\theta} \psi, \psi^{(j + \sigma)}) =
(\delta^{\alpha, \beta}_{h} \psi^{(j + \sigma)}, \psi^{(j + \sigma)})
+ (y(u^{j + \sigma}) - y(U^{j + \sigma}), \psi^{(j + \sigma)}) + (\tilde{R}_{i}^{j + \sigma}, \psi^{(j + \sigma)}).
\end{equation*}
Similar as the proof of Theorem 2.2, we arrive at
\begin{equation*}
\begin{split}
& \left\{ \frac{c_{0}^{(\theta, \sigma)}}{2 \tau^{\theta} \Gamma(2 - \theta)}
- L_1 \left[ \sigma^2 + (1 - \sigma)^2 \right] \right\} \left \| \psi^{j + 1} \right \|^2 \\
&\quad \leq
L_1 \left[ \sigma^2 + (1 - \sigma)^2 \right] \left \| \psi^{j} \right \|^2
+ \frac{1}{4 c \ln 2} \left \|  \tilde{R}^{j + \sigma} \right \|^2 \\
&\qquad\quad + \frac{1}{2 \tau^{\theta} \Gamma(2 - \theta)} \sum_{s = 1}^{j} \left( c_{j - s}^{(\theta, \sigma)}
- c_{j - s + 1}^{(\theta, \sigma)} \right) \left \| \psi^{s} \right \|^2.
\end{split}
\end{equation*}
Utilizing the mathematical induction method, the priori estimate follows that
\begin{equation*}
\left \| \psi^j \right \| \leq \hat{C} (\tau^2 + h^2), ~~ 0 \leq j \leq M.
\end{equation*}
The proof of this theorem is completed.
\end{proof}
The two assumptions about $y(u(x, t))$ are extremely strict, which is not very satisfactory.
So our further work is to seek a more suitable method to prove the stability of IDS (\ref{eq2.10}).
\section{Fast implementation of IDS with the circulant precondtioner}
\label{sec3} \quad\

Before moving into the investigation of fast solution techniques,
the matrix form of implicit difference scheme (\ref{eq2.4}) must be established first.
In order to facilitate our discussion, we use notations in Section \ref{sec2}.
Then, the matrix form of (\ref{eq2.4}) corresponding to each time layer $j$ can be written as follows:
\begin{equation}
\mathcal{M}^{j + \sigma} u^{j + 1}
 = B^{j + \sigma} u^{j}
 - \frac{\tau^{-\theta}}{\Gamma(2 - \theta)} \sum_{s = 0}^{j - 1} c_{j - s}^{(\theta,\sigma)}
( u^{s + 1} - u^{s})  + f^{j + \sigma}, \quad 0 \leq j \leq M - 1,
\label{eq3.1}
\end{equation}
we assume that the sums to be equal to zero if the upper summation index is less than the lower one,
where
\begin{equation*}
\mathcal{M}^{j + \sigma} =
\frac{\tau^{-\theta}}{\Gamma(2 - \theta)} c_{0}^{(\theta,\sigma)} I
+ \sigma \Big( \frac{d_{+}^{j + \sigma}}{h^{\alpha}} A_{\alpha}
+ \frac{d_{-}^{j + \sigma}}{h^{\alpha}} A_{\alpha}^{T}
- \frac{e_{+}^{j + \sigma}}{h^{\beta}} A_{\beta}
- \frac{e_{-}^{j + \sigma}}{h^{\beta}} A_{\beta}^{T} \Big),
\end{equation*}
\begin{equation*}
B^{j + \sigma} =
\frac{\tau^{-\theta}}{\Gamma(2 - \theta)} c_{0}^{(\theta,\sigma)} I
- (1 - \sigma) \Big( \frac{d_{+}^{j + \sigma}}{h^{\alpha}} A_{\alpha}
+ \frac{d_{-}^{j + \sigma}}{h^{\alpha}} A_{\alpha}^{T}
- \frac{e_{+}^{j + \sigma}}{h^{\beta}} A_{\beta}
- \frac{e_{-}^{j + \sigma}}{h^{\beta}} A_{\beta}^{T} \Big),
\end{equation*}
and
$I$ is an identity matrix of order $N - 1$, the two matrices $A_{\alpha}$
and $A_{\beta}$ are defined by
\vspace{4mm}

$A_{\alpha} =
\begin{bmatrix}
 \omega_1^{(\alpha)} &  \omega_0^{(\alpha)} &  &  & \\
 \omega_2^{(\alpha)} & \omega_1^{(\alpha)} &  \omega_0^{(\alpha)} &  & \\
 \vdots & \ddots & \ddots &  \ddots & \\
 \omega_{N-2}^{(\alpha)} & \cdots & \ddots & \ddots & \omega_0^{(\alpha)} \\
 \omega_{N-1}^{(\alpha)} & \omega_{N-2}^{(\alpha)} & \cdots &  \omega_2^{(\alpha)} & \omega_1^{(\alpha)}
\end{bmatrix}$, ~~~
$A_{\beta} =
\begin{bmatrix}
 \omega_1^{(\beta)} &  \omega_0^{(\beta)} &  &  & \\
 \omega_2^{(\beta)} & \omega_1^{(\beta)} &  \omega_0^{(\beta)} &  & \\
 \vdots & \ddots & \ddots &  \ddots & \\
 \omega_{N-2}^{(\beta)} & \cdots & \ddots & \ddots & \omega_0^{(\beta)} \\
 \omega_{N-1}^{(\beta)} & \omega_{N-2}^{(\beta)} & \cdots &  \omega_2^{(\beta)} & \omega_1^{(\beta)}
\end{bmatrix}$.
\vspace{4mm}

It is apparent that $A_{\alpha}$ and $A_{\beta}$ are Toeplitz matrices.
Therefore, they can be stored with $2N+2$ entries.
Krylov subspace methods with suitable circulant preconditioners \cite{ng2001itertoep,hwsun2013circulant}
can be used to efficiently solve Toeplitz or Toeplitz-like linear systems with a fast convergence rate. In this case,
it also remarked that the algorithmic complexity of preconditioned Krylov subspace methods is only in
$\mathcal{O}(N \log N)$ arithmetic operations per iteration step.

Inspired by the above consideration, we propose a circulant preconditioner to solve (\ref{eq3.1}), which is generated from
the Strang's circulant preconditioner \cite{chan2007toep}, through two preconditioned Krylov subspace methods.
The Strang's circulant matrix $s(Q) = [s_{j-k}]_{0 \leq j, k< N}$
for a real Toeplitz matrix $Q = [q_{j-k}]_{0 \leq j, k< N}$ is obtained by copying the central
diagonals of $Q$ and bringing them around to complete the circulant requirement.
More precisely, the diagonals of $s(Q)$ are given by
\begin{equation*}
s_j =
\begin{cases}
q_j, & 0 \leq j < N/2, \\
0, & j = N/2~if~N~is~even, \\
q_{j - N}, & N/2 < j < N, \\
s_{j + N}, & 0 < -j < N.
\end{cases}
\end{equation*}
Then our circulant preconditioner is defined as
\begin{equation}
P^{(j + \sigma)} =
\begin{cases}
\begin{split}
\frac{\tau^{-\theta}}{\Gamma(2 - \theta)} & a_{0}^{(\theta,\sigma)} I
+ \sigma \Big( \frac{d_{+}^{j + \sigma}}{h^{\alpha}} s(A_{\alpha})
+ \frac{d_{-}^{j + \sigma}}{h^{\alpha}} s(A_{\alpha}^{T}) \\
& - \frac{e_{+}^{j + \sigma}}{h^{\beta}} s(A_{\beta})
- \frac{e_{-}^{j + \sigma}}{h^{\beta}} s(A_{\beta}^{T}) \Big), \quad j = 0,
\end{split}\\
\begin{split}
\frac{\tau^{-\theta}}{\Gamma(2 - \theta)} (& a_{0}^{(\theta,\sigma)}
+ b_{1}^{(\theta,\sigma)}) I
+ \sigma \Big( \frac{d_{+}^{j + \sigma}}{h^{\alpha}} s(A_{\alpha})
+ \frac{d_{-}^{j + \sigma}}{h^{\alpha}} s(A_{\alpha}^{T}) \\
& - \frac{e_{+}^{j + \sigma}}{h^{\beta}} s(A_{\beta})
- \frac{e_{-}^{j + \sigma}}{h^{\beta}} s(A_{\beta}^{T}) \Big), \quad j = 1,\cdots,M-1,
\end{split}
\end{cases}
\label{eq3.2}
\end{equation}
where the first columns of $s(A_{\gamma})$ and $s(A_{\gamma}^{T})$ ($\gamma = \alpha, \beta$)
are given by
\begin{center}
$\begin{bmatrix}
\omega_1^{(\gamma)} \\
\vdots \\
\omega_{\left \lfloor \frac{N} {2}\right \rfloor}^{(\gamma)}\\
0 \\
\vdots \\
0 \\
\omega_0^{(\gamma)}
\end{bmatrix}$
and
$\begin{bmatrix}
\omega_1^{(\gamma)} \\
\omega_0^{(\gamma)} \\
0 \\
\vdots \\
0 \\
\omega_{\left \lfloor \frac{N} {2}\right \rfloor}^{(\gamma)}\\
\vdots \\
\omega_2^{(\gamma)}
\end{bmatrix}$, respectively.
\end{center}

To make sure the preconditioner defined in (\ref{eq3.2}) is well-defined,
let us illustrate that $P^{(j + \sigma)}$ are nonsingular.
Before that, we need to give the following theorem, which is essential to check the nonsingularity of $P^{(j + \sigma)}$.
\begin{theorem}
1. The real parts of all eigenvalues of $s(A_{\alpha})$ and
$s(A_{\alpha}^{T})$ are strictly positive for all $N$;

2. The real parts of all eigenvalues of $s(A_{\beta})$ and
$s(A_{\beta}^{T})$ are strictly negative for all $N$.
\label{th3.1}
\end{theorem}
\begin{proof}
Firstly, we proof the real parts of
all eigenvalues of $s(A_{\alpha})$ are strictly positive for all
$N$.

Recall that the real parts of all eigenvalues of
$s(A_{\alpha})$ are equivalent to the eigenvalues of
$\frac{s(A_{\alpha}) + [s(A_{\alpha})]^{*}}{2}$
(where $S^{*}$ represents conjugate transpose of matrix $S$).
Therefore, it is sufficient to show that all eigenvalues of
$\frac{s(A_{\alpha}) + s^{*}(A_{\alpha})}{2}$ are
strictly positive for all $N$.

According to the Gershgorin circle theorem,
all the Gershgorin disc of the circulant matrix
$\frac{s(A_{\alpha}) + s^{*}(A_{\alpha})}{2}$
are centered at $\omega_1^{(\alpha)} > 0$ with radius
\begin{equation*}
r_{\alpha} = -\sum\limits_{k = 0,k \neq 1}^{\left \lfloor \frac{N} {2}\right \rfloor} \omega_{k}^{(\alpha)}
< -\sum\limits_{k = 0,k \neq 1}^{\infty} \omega_{k}^{(\alpha)}
= \omega_1^{(\alpha)},
\end{equation*}
by the Lemma \ref{lemma2.2}, the first part of this
theorem is proved.

Then, we prove the second part of the theorem.
Similarly, all the Gershgorin disc of the circulant matrix
$\frac{s(A_{\beta}) + s^{*}(A_{\beta})}{2}$
are centered at $\omega_1^{(\beta)} < 0$ with radius
\begin{equation*}
r_{\beta} = \sum\limits_{k = 0,k \neq 1}^{\left \lfloor \frac{N} {2}\right \rfloor} \omega_{k}^{(\beta)}
< \sum\limits_{k = 0,k \neq 1}^{\infty} \omega_{k}^{(\beta)}
= -\omega_1^{(\beta)},
\end{equation*}
by the Lemma \ref{lemma2.3}, the second part of this
theorem is proved. Finally, we complete the proof of the
theorem.
\end{proof}
With the help of Theorem \ref{th3.1}, we will check the circulant preconditioners $P^{(j + \sigma)}$ are
nonsingular.
\begin{theorem}
The circulant preconditioners $P^{(j + \sigma)}$ defined in (\ref{eq3.2}) are nonsingular.
\end{theorem}
\begin{proof}
As we know, a circulant matrix can be diagonalized by the Fourier matrix F \cite{chan2007toep}. Then it follows that
$s(A_{\gamma}) = F^{*} \Lambda_{\gamma} F$, $s(A_{\gamma}^{T}) = F^{*} \bar{\Lambda}_{\gamma} F$,
where $\bar{\Lambda}_{\gamma}$ is the complex conjugate of $\Lambda_{\gamma}$. Decompose the circulant matrix
$P^{(j + \sigma)} = F^{*} \Lambda_{P} F$ with the diagonal matrix
\begin{equation*}
\Lambda_{P} =
\begin{cases}
\begin{split}
\frac{\tau^{-\theta}}{\Gamma(2 - \theta)} & a_{0}^{(\theta,\sigma)} I
+ \sigma \Big( \frac{d_{+}^{j + \sigma}}{h^{\alpha}} \Lambda_{\alpha}
+ \frac{d_{-}^{j + \sigma}}{h^{\alpha}} \bar{\Lambda}_{\alpha} \\
& - \frac{e_{+}^{j + \sigma}}{h^{\beta}} \Lambda_{\beta}
- \frac{e_{-}^{j + \sigma}}{h^{\beta}} \bar{\Lambda}_{\beta} \Big), \quad j = 0,
\end{split}\\
\begin{split}
\frac{\tau^{-\theta}}{\Gamma(2 - \theta)} (& a_{0}^{(\theta,\sigma)}
+ b_{1}^{(\theta,\sigma)}) I
+ \sigma \Big( \frac{d_{+}^{j + \sigma}}{h^{\alpha}} \Lambda_{\alpha}
+ \frac{d_{-}^{j + \sigma}}{h^{\alpha}} \bar{\Lambda}_{\alpha} \\
& - \frac{e_{+}^{j + \sigma}}{h^{\beta}} \Lambda_{\beta}
- \frac{e_{-}^{j + \sigma}}{h^{\beta}} \bar{\Lambda}_{\beta} \Big), \quad j = 1,\cdots,M-1.
\end{split}
\end{cases}
\end{equation*}

Then the real part of $\Lambda_P$ can be written as
\begin{equation*}
Re([\Lambda_{P}]_{k,k}) =
\begin{cases}
\begin{split}
\frac{\tau^{-\theta}}{\Gamma(2 - \theta)} & a_{0}^{(\theta,\sigma)}
+ \sigma \Big( \frac{d_{+}^{j + \sigma}}{h^{\alpha}} Re([\Lambda_{\alpha}]_{k,k})
+ \frac{d_{-}^{j + \sigma}}{h^{\alpha}} Re([\bar{\Lambda}_{\alpha}]_{k,k}) \\
& - \frac{e_{+}^{j + \sigma}}{h^{\beta}} Re([\Lambda_{\beta}]_{k,k})
- \frac{e_{-}^{j + \sigma}}{h^{\beta}} Re([\bar{\Lambda}_{\beta}]_{k,k}) \Big), \quad j = 0,
\end{split}\\
\begin{split}
\frac{\tau^{-\theta}}{\Gamma(2 - \theta)} (& a_{0}^{(\theta,\sigma)}
+ b_{1}^{(\theta,\sigma)})
+ \sigma \Big( \frac{d_{+}^{j + \sigma}}{h^{\alpha}} Re([\Lambda_{\alpha}]_{k,k})
+ \frac{d_{-}^{j + \sigma}}{h^{\alpha}} Re([\bar{\Lambda}_{\alpha}]_{k,k}) \\
& - \frac{e_{+}^{j + \sigma}}{h^{\beta}} Re([\Lambda_{\beta}]_{k,k})
- \frac{e_{-}^{j + \sigma}}{h^{\beta}} Re([\bar{\Lambda}_{\beta}]_{k,k}) \Big), \quad j = 1,\cdots,M-1.
\end{split}
\end{cases}
\end{equation*}
Combining Theorem \ref{th3.1}, we obtain
\begin{equation*}
Re([\Lambda_P]_{k,k}) > 0.
\end{equation*}

Consequently, $P^{(j + \sigma)}$ are invertible.
\end{proof}

Unfortunately, due to the properties of coefficients $\omega_{k}^{(\alpha)}$ and $\omega_{k}^{(\beta)}$,
it is difficult to theoretically investigate the eigenvalue distributions of
preconditioned matrix $(P^{(j + \sigma)})^{-1} \mathcal{M}^{j + \sigma}$,
but we still can give some figures to illustrate the
clustering eigenvalue distributions of several specified preconditioned matrices in section 4.

At the end of this section, we study the fast implementation of IDS (\ref{eq2.10}).
For convenience, (\ref{eq2.10}) is recast as
\begin{equation}
\mathcal{M}^{j + \sigma} u^{j + 1}
 = B^{j + \sigma} u^{j}
 - \frac{\tau^{-\theta}}{\Gamma(2 - \theta)} \sum_{s = 0}^{j - 1} c_{j - s}^{(\theta,\sigma)}
( u^{s + 1} - u^{s}) + F^{j + 1}, \quad 0 \leq j \leq M - 1,
\label{eq3.3}
\end{equation}
in which
\begin{equation*}
F^{j + 1} = \left[ g(\sigma u_{1}^{j + 1} + (1 - \sigma) u_{1}^{j}, x_1, t_{j + \sigma}),
\cdots, g(\sigma u_{N - 1}^{j + 1} + (1 - \sigma) u_{N - 1}^{j}, x_{N - 1}, t_{j + \sigma}) \right]^{T}.
\end{equation*}
Then the following algorithm is utilized to solve (\ref{eq3.3}),
\begin{equation}
\begin{split}
\mathcal{M}^{j + \sigma} u^{j + 1 (l + 1)}
 = B^{j + \sigma} u^{j}
 - \frac{\tau^{-\theta}}{\Gamma(2 - \theta)} \sum_{s = 0}^{j - 1} c_{j - s}^{(\theta,\sigma)}
& ( u^{s + 1} - u^{s}) + F^{j + 1 (l)}, \\
& 0 \leq j \leq M - 1, l = 0,1,2,\cdots
\end{split}
\label{eq3.4}
\end{equation}
with
\begin{equation}
u_{0}^{j (l)} = u_{N}^{j (l)} = 0, \quad
u^{j + 1 (0)} =
\begin{cases}
u^{j}, & j = 0, \\
2 u^{j} - u^{j - 1}, & j \geq 1
\end{cases}
\label{eq3.5}
\end{equation}
and
\begin{equation}
\begin{split}
F^{j + 1 (l)} = & \Big[ g(\sigma u_{1}^{j + 1(l)} + (1 - \sigma) u_{1}^{j}, x_1, t_{j + \sigma}), \\
&\quad \cdots, g(\sigma u_{N - 1}^{j + 1(l)} + (1 - \sigma) u_{N - 1}^{j}, x_{N - 1}, t_{j + \sigma}) \Big]^{T}.
\end{split}
\label{eq3.6}
\end{equation}
At each iteration, the above system is linearized,
and compared with (\ref{eq3.1}), the system has not changed significantly.
Thus, the previously mentioned fast method can be applied directly to this system.
In the next section, we provide two experiments to test the effectiveness of our fast method, even for nonlinear situation.
\section{Numerical results}
\label{sec5} \quad\

In this section, we carry out several numerical experiments,
which are given in Examples 1-2, to illustrate that
our proposed IDM can indeed convergent with second order accuracy in both space
and time. Some other numerical experiments, given in Examples 3-4, are
reported to illustrate the effectiveness of the fast solution techniques.
For direct solver, we choose LU factorization of MATLAB in Examples 3-4,
and its CPU time is represented.
For two Krylov subspace methods (BiCGSTAB and GPBiCOR($m$, $\ell$)) with circulant preconditioners,
number of iterations required for convergence and CPU time of those methods are reported.
We denote PBiCGSTAB and PGPBiCOR($m$, $\ell$) as the preconditioned version of
BiCGSTAB and GPBiCOR($m$, $\ell$) respectively. The stopping criterion of those methods is
\begin{equation*}
\frac{\left \| r^{(k)} \right \|_2}{\left \| r^{(0)} \right \|_2} < 10^{-12},
\end{equation*}
where $r^{(k)}$ is the residual vector of the linear system after $k$ iterations,
and the initial guess at each time step is chosen as the zero vector.
All experiments were performed on a Windows 7 (32 bit) PC-Intel(R)
Core(TM) i3-2130 CPU 3.40GHz, 4GB of RAM using MATLAB R2015b.
\vspace{4mm}

\subsection{Verification of convergence rate}
\label{sec4.1} \quad\

\noindent{\textbf{Example 1.}} We consider the linear problem  (\ref{eq1.1}), take coefficients
\begin{equation*}
\begin{split}
& d_{+}(t) = \exp(t), \qquad d_{-}(t) = 3\exp(-t), \\
& e_{+}(t) = (1 + t)^2, \qquad e_{-}(t) = 1 + t^2,
\end{split}
\end{equation*}
and the source term
\begin{equation*}
\begin{split}
f(x,t) = & \frac{\Gamma(3 + \theta)}{\Gamma(3)} t^2 x^2 (1 - x)^2
+ t^{\theta + 2} \Big\{ \frac{\Gamma(3)}{\Gamma(3 - \alpha)} \left[ d_{+}(t) x^{2 - \alpha} + d_{-}(t) (1 - x)^{2 - \alpha} \right] \\
&~ - \frac{2 \Gamma(4)}{\Gamma(4 - \alpha)} \left[ d_{+}(t) x^{3 - \alpha} + d_{-}(t) (1 - x)^{3 - \alpha} \right]
+ \frac{\Gamma(5)}{\Gamma(5 - \alpha)} \left[ d_{+}(t) x^{4 - \alpha} + d_{-}(t) (1 - x)^{4 - \alpha} \right] \\
&~ - \frac{\Gamma(3)}{\Gamma(3 - \beta)} \left[ e_{+}(t) x^{2 - \beta} + e_{-}(t) (1 - x)^{2 - \beta} \right]
+ \frac{2 \Gamma(4)}{\Gamma(4 - \beta)} \left[ e_{+}(t) x^{3 - \beta} + e_{-}(t) (1 - x)^{3 - \beta} \right] \\
&\quad - \frac{\Gamma(5)}{\Gamma(5 - \beta)} \left[ e_{+}(t) x^{4 - \beta} + e_{-}(t) (1 - x)^{4 - \beta} \right] \Big\}.
\end{split}
\end{equation*}
Then the causal solution is $u(x,t) = t^{\theta  + 2} x^2 (1 - x)^2$.

As can be seen in Tables \ref{tab1} and \ref{tab3}, when $h = 1/3000$,
the maximum error decreases steadily with the shortening of time step,
and the convergence order of time is the expected $\mathcal{O} (\tau^2)$,
where the convergence order (CO) is given by the
following formula: $\rm{CO} = \log_{\tau_1/\tau_2} \frac{\left \| \xi_1 \right \|}{\left \| \xi_2 \right \|}$
($\xi_i$ is the error corresponding to $h_i$).
On the other hand,
Tables \ref{tab2} and \ref{tab4} illustrate that if $h = \tau$,
a reduction in the maximum error occurs along with the decrease of space step and time step,
and the spatial convergence order is $\mathcal{O} (h^2)$,
where the convergence order (CO) is given by the formula:
$\rm{CO} =\log_{h_1/h_2} \frac{\left \| \xi_1 \right \|}{\left \| \xi_2 \right \|}$.
Figs. \ref{fig:1}-\ref{fig:2} are plotted to further explain the
reliability of our proposed scheme, and indicate that the `quadratic-type' order of accuracy can achieve
the desired order $\mathcal{O}(\tau^2 + h^2)$.

\begin{table}[h]
\caption{$L_2$-norm and maximum norm error behavior versus
$\tau$-grid size reduction when $\alpha = 0.6$, $\beta = 1.8$, and $h = 1/3000$ in Example 1.}
\centering
\begin{tabular}{cccccc}
\hline
$\theta$ & $\tau$ & $ \left \| \xi \right \|_{\mathcal{C}(\bar{\omega}_{h \tau})}$ & CO in $ \left \| \cdot \right \|_{\mathcal{C}(\bar{\omega}_{h \tau})} $
& max$_{0 \leq n \leq M} \left \| \xi^n \right \|_0$ & CO in $\left \| \cdot \right \|_0$ \\
\hline
0.10 & 1/8 & 5.2321e-05 & -- & 3.3293e-05 & -- \\
    & 1/16 & 1.3096e-05 & 1.9982 & 8.3340e-06 & 1.9981 \\
    & 1/32 & 3.2603e-06 & 2.0061 & 2.0748e-06 & 2.0061 \\
    & 1/64 & 7.9868e-07 & 2.0293 & 5.0828e-07 & 2.0292 \\
    & 1/128 & 1.8626e-07 & 2.1003 & 1.1850e-07 & 2.1007 \\
\hline
0.50 & 1/8 & 3.3405e-04 & -- & 2.1261e-04 & -- \\
    & 1/16 & 8.3929e-05 & 1.9928 & 5.3418e-05 & 1.9928 \\
    & 1/32 & 2.1014e-05 & 1.9978 & 1.3374e-05 & 1.9979 \\
    & 1/64 & 5.2426e-06 & 2.0030 & 3.3366e-06 & 2.0030 \\
    & 1/128 & 1.2947e-06 & 2.0177 & 8.2404e-07 & 2.0176 \\
\hline
0.90 & 1/8 & 6.5352e-04 & -- & 4.1601e-04 & -- \\
    & 1/16 & 1.6365e-04 & 1.9976 & 1.0417e-04 & 1.9976 \\
    & 1/32 & 4.0921e-05 & 1.9997 & 2.6048e-05 & 1.9997 \\
    & 1/64 & 1.0216e-05 & 2.0020 & 6.5027e-06 & 2.0021 \\
    & 1/128 & 2.5376e-06 & 2.0093 & 1.6153e-06 & 2.0093 \\
\hline
0.99 & 1/8 & 7.1584e-04 & -- & 4.5567e-04 & -- \\
    & 1/16 & 1.7899e-04 & 1.9998 & 1.1395e-04 & 1.9996 \\
    & 1/32 & 4.4733e-05 & 2.0004 & 2.8479e-05 & 2.0005 \\
    & 1/64 & 1.1167e-05 & 2.0021 & 7.1095e-06 & 2.0021 \\
    & 1/128 & 2.7756e-06 & 2.0084 & 1.7671e-06 & 2.0084 \\
\hline
\end{tabular}
\label{tab1}
\end{table}
\begin{table}[h]
\caption{$L_2$-norm and maximum norm error behavior versus
grid size reduction when $\alpha = 0.6$, $\beta = 1.8$, and $\tau = h$ in Example 1.}
\centering
\begin{tabular}{cccccc}
\hline
$\theta$ & $\tau$ & $ \left \| \xi \right \|_{\mathcal{C}(\bar{\omega}_{h \tau})}$ & CO in $ \left \| \cdot \right \|_{\mathcal{C}(\bar{\omega}_{h \tau})} $
& max$_{0 \leq n \leq M} \left \| \xi^n \right \|_0$ & CO in $\left \| \cdot \right \|_0$ \\
\hline
0.10 & 1/20 & 6.0326e-04 & -- & 4.1160e-04 & -- \\
    & 1/40 & 1.4719e-04 & 2.0351 & 9.9593e-05 & 2.0471 \\
    & 1/80 & 3.5781e-05 & 2.0404 & 2.4068e-05 & 2.0489 \\
    & 1/160 & 8.6941e-06 & 2.0411 & 5.8187e-06 & 2.0483 \\
    & 1/320 & 2.1128e-06 & 2.0409 & 1.4087e-06 & 2.0464 \\
\hline
0.50 & 1/20 & 5.5002e-04 & -- & 3.7743e-04 & -- \\
    & 1/40 & 1.3407e-04 & 2.0365 & 9.1172e-05 & 2.0496 \\
    & 1/80 & 3.2522e-05 & 2.0435 & 2.1985e-05 & 2.0520 \\
    & 1/160 & 7.8829e-06 & 2.0446 & 5.3038e-06 & 2.0514 \\
    & 1/320 & 1.9110e-06 & 2.0444 & 1.2815e-06 & 2.0492 \\
\hline
0.90 & 1/20 & 4.8566e-04 & -- & 3.3574e-04 & -- \\
    & 1/40 & 1.1823e-04 & 2.0384 & 8.0954e-05 & 2.0522 \\
    & 1/80 & 2.8599e-05 & 2.0475 & 1.9472e-05 & 2.0557 \\
    & 1/160 & 6.9105e-06 & 2.0491 & 4.6857e-06 & 2.0551 \\
    & 1/320 & 1.6699e-06 & 2.0490 & 1.1296e-06 & 2.0524 \\
\hline
0.99 & 1/20 & 4.7183e-04 & -- & 3.2657e-04 & -- \\
    & 1/40 & 1.1479e-04 & 2.0392 & 7.8721e-05 & 2.0526 \\
    & 1/80 & 2.7752e-05 & 2.0484 & 1.8925e-05 & 2.0564 \\
    & 1/160 & 6.7015e-06 & 2.0500 & 4.5520e-06 & 2.0558 \\
    & 1/320 & 1.6181e-06 & 2.0502 & 1.0969e-06 & 2.0530 \\
\hline
\end{tabular}
\label{tab2}
\end{table}
\begin{figure}[!htbp]
\centering
\includegraphics[width=3.12in,height=3.0in]{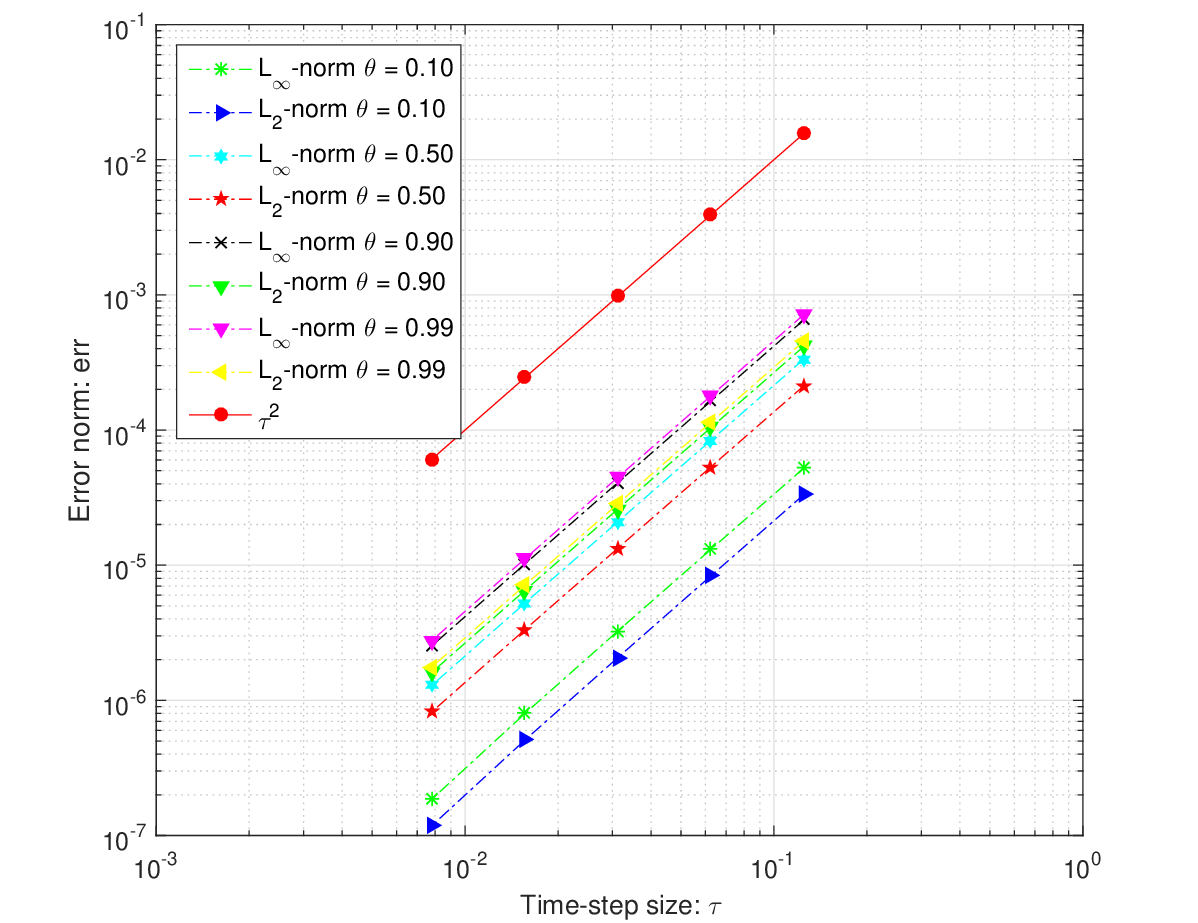}
\includegraphics[width=3.12in,height=3.0in]{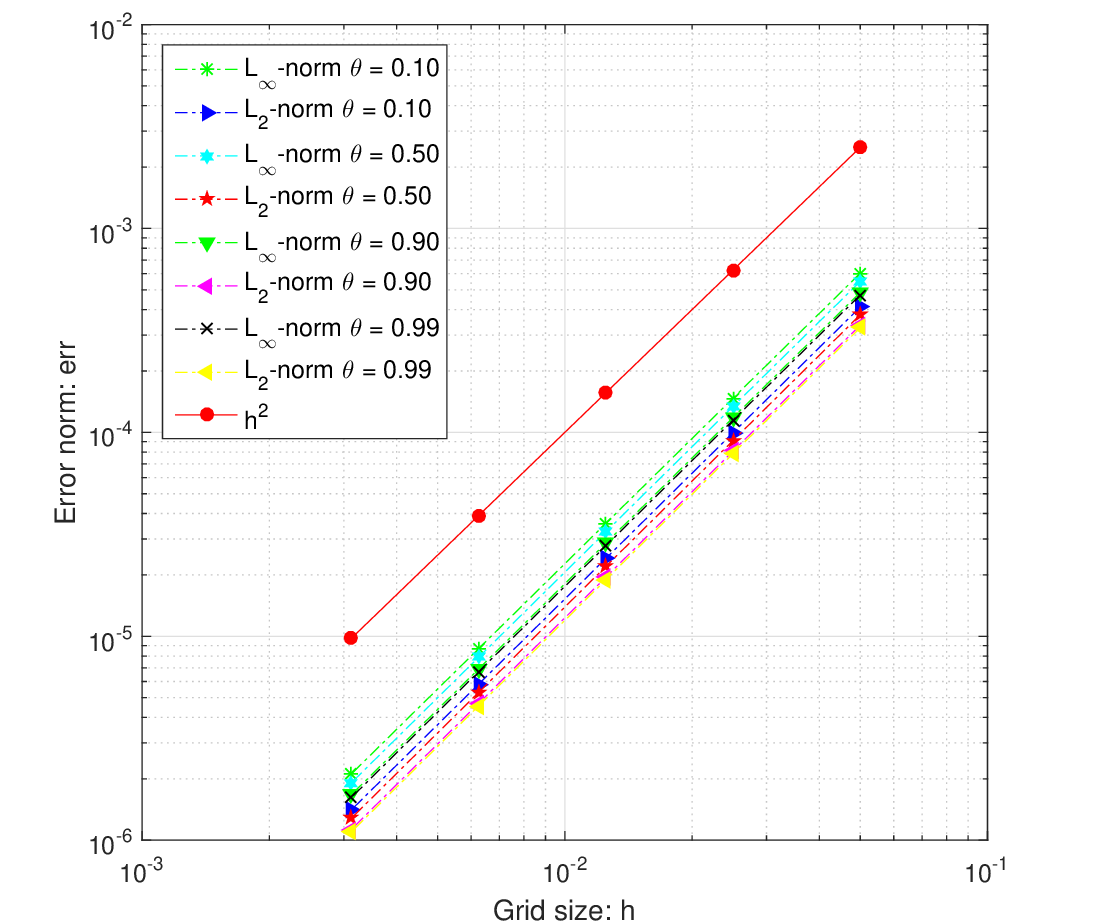}
\caption{Comparison the order of accuracy obtained by our proposed IDS for Example 1 in
space and time variables. Left: time direction; Right: space direction.}
\label{fig:1}
\end{figure}
\begin{table}[h]
\caption{$L_2$-norm and maximum norm error behavior versus
$\tau$-grid size reduction when $\alpha = 0.99$, $\beta = 1.99$, and $h = 1/3000$ in Example 1.}
\centering
\begin{tabular}{cccccc}
\hline
$\theta$ & $\tau$ & $ \left \| \xi \right \|_{\mathcal{C}(\bar{\omega}_{h \tau})}$ & CO in $ \left \| \cdot \right \|_{\mathcal{C}(\bar{\omega}_{h \tau})} $
& max$_{0 \leq n \leq M} \left \| \xi^n \right \|_0$ & CO in $\left \| \cdot \right \|_0$ \\
\hline
0.10 & 1/8 & 5.2548e-05 & -- & 3.3453e-05 & -- \\
    & 1/16 & 1.3149e-05 & 1.9987 & 8.3697e-06 & 1.9989 \\
    & 1/32 & 3.2701e-06 & 2.0076 & 2.0799e-06 & 2.0086 \\
    & 1/64 & 7.9770e-07 & 2.0354 & 5.0584e-07 & 2.0398 \\
    & 1/128 & 1.8525e-07 & 2.1064 & 1.1708e-07 & 2.1112 \\
\hline
0.50 & 1/8 & 3.3536e-04 & -- & 2.1354e-04 & -- \\
    & 1/16 & 8.4257e-05 & 1.9929 & 5.3650e-05 & 1.9929 \\
    & 1/32 & 2.1094e-05 & 1.9980 & 1.3430e-05 & 1.9982 \\
    & 1/64 & 5.2595e-06 & 2.0038 & 3.3469e-06 & 2.0045 \\
    & 1/128 & 1.2956e-06 & 2.0213 & 8.2291e-07 & 2.0240 \\
\hline
0.90 & 1/8 & 6.5583e-04 & -- & 4.1765e-04 & -- \\
    & 1/16 & 1.6423e-04 & 1.9976 & 1.0458e-04 & 1.9976 \\
    & 1/32 & 4.1065e-05 & 1.9997 & 2.6150e-05 & 1.9998 \\
    & 1/64 & 1.0249e-05 & 2.0024 & 6.5248e-06 & 2.0028 \\
    & 1/128 & 2.5427e-06 & 2.0111 & 1.6172e-06 & 2.0124 \\
\hline
0.99 & 1/8 & 7.1823e-04 & -- & 4.5737e-04 & -- \\
    & 1/16 & 1.7958e-04 & 1.9998 & 1.1437e-04 & 1.9997 \\
    & 1/32 & 4.4877e-05 & 2.0005 & 2.8580e-05 & 2.0006 \\
    & 1/64 & 1.1200e-05 & 2.0025 & 7.1310e-06 & 2.0028 \\
    & 1/128 & 2.7804e-06 & 2.0101 & 1.7688e-06 & 2.0113 \\
\hline
\end{tabular}
\label{tab3}
\end{table}
\begin{table}[h]
\caption{$L_2$-norm and maximum norm error behavior versus
grid size reduction when $\alpha = 0.99$, $\beta = 1.99$, and $\tau = h$ in Example 1.}
\centering
\begin{tabular}{cccccc}
\hline
$\theta$ & $\tau$ & $ \left \| \xi \right \|_{\mathcal{C}(\bar{\omega}_{h \tau})}$
& CO in $ \left \| \cdot \right \|_{\mathcal{C}(\bar{\omega}_{h \tau})} $
& max$_{0 \leq n \leq M} \left \| \xi^n \right \|_0$ & CO in $\left \| \cdot \right \|_0$ \\
\hline
0.10 & 1/20 & 6.0735e-04 & -- & 4.4270e-04 & -- \\
    & 1/40 & 1.5155e-04 & 2.0027 & 1.1041e-04 & 2.0035 \\
    & 1/80 & 3.7819e-05 & 2.0026 & 2.7528e-05 & 2.0039 \\
    & 1/160 & 9.4336e-06 & 2.0032 & 6.8625e-06 & 2.0041 \\
    & 1/320 & 2.3531e-06 & 2.0033 & 1.7107e-06 & 2.0041 \\
\hline
0.50 & 1/20 & 5.5527e-04 & -- & 4.0960e-04 & -- \\
    & 1/40 & 1.3860e-04 & 2.0022 & 1.0219e-04 & 2.0030 \\
    & 1/80 & 3.4589e-05 & 2.0026 & 2.5475e-05 & 2.0041 \\
    & 1/160 & 8.6263e-06 & 2.0035 & 6.3494e-06 & 2.0044 \\
    & 1/320 & 2.1512e-06 & 2.0036 & 1.5825e-06 & 2.0045 \\
\hline
0.90 & 1/20 & 4.9287e-04 & -- & 3.6976e-04 & -- \\
    & 1/40 & 1.2317e-04 & 2.0006 & 9.2308e-05 & 2.0021 \\
    & 1/80 & 3.0734e-05 & 2.0027 & 2.3012e-05 & 2.0041 \\
    & 1/160 & 7.6637e-06 & 2.0037 & 5.7346e-06 & 2.0046 \\
    & 1/320 & 1.9108e-06 & 2.0039 & 1.4289e-06 & 2.0048 \\
\hline
0.99 & 1/20 & 4.7956e-04 & -- & 3.6117e-04 & -- \\
    & 1/40 & 1.1989e-04 & 2.0000 & 9.0184e-05 & 2.0017 \\
    & 1/80 & 2.9914e-05 & 2.0028 & 2.2483e-05 & 2.0041 \\
    & 1/160 & 7.4589e-06 & 2.0038 & 5.6023e-06 & 2.0047 \\
    & 1/320 & 1.8596e-06 & 2.0040 & 1.3958e-06 & 2.0049 \\
\hline
\end{tabular}
\label{tab4}
\end{table}
\begin{figure}[H]
\centering
\includegraphics[width=3.12in,height=3.0in]{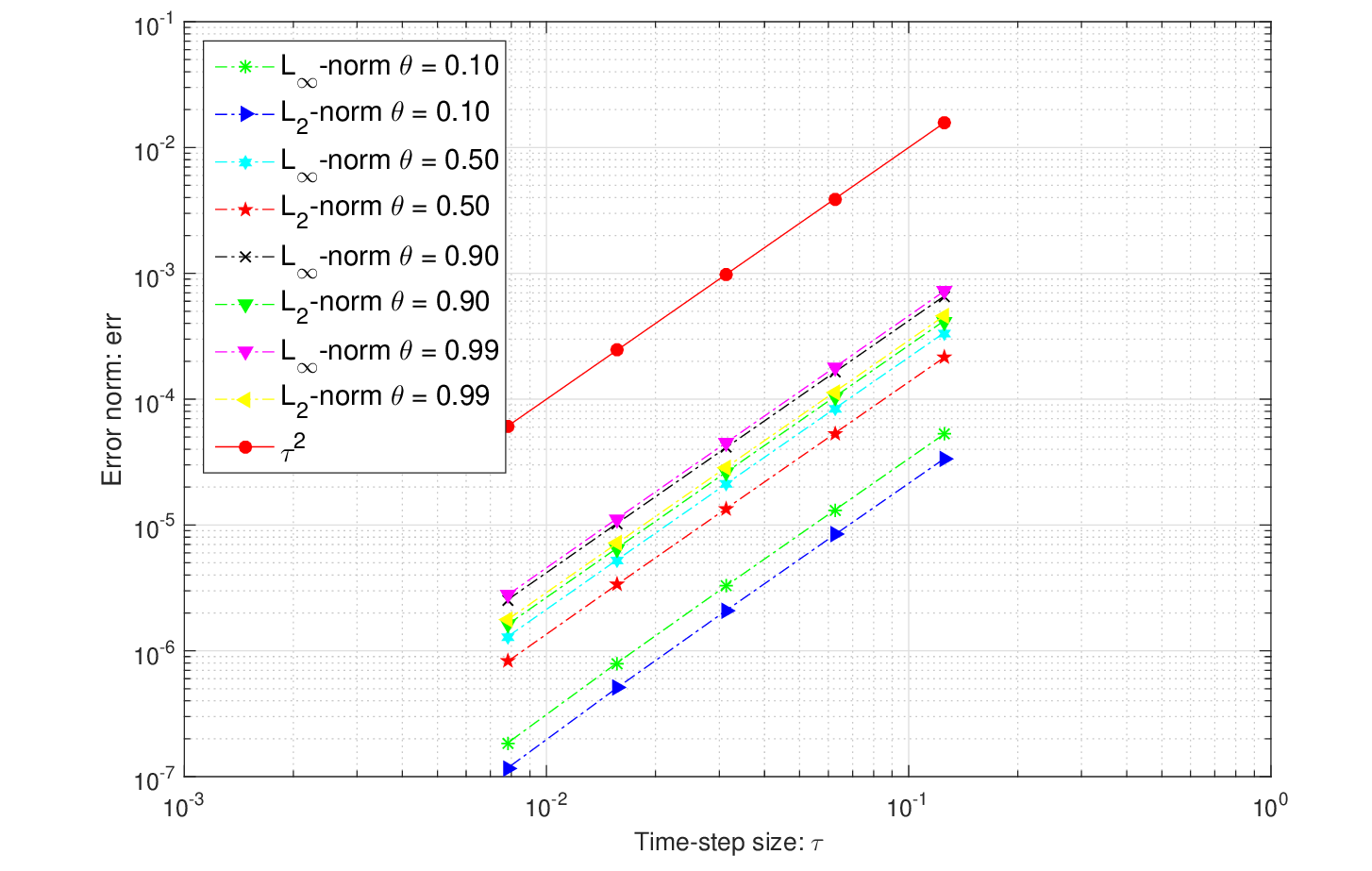}
\includegraphics[width=3.12in,height=3.0in]{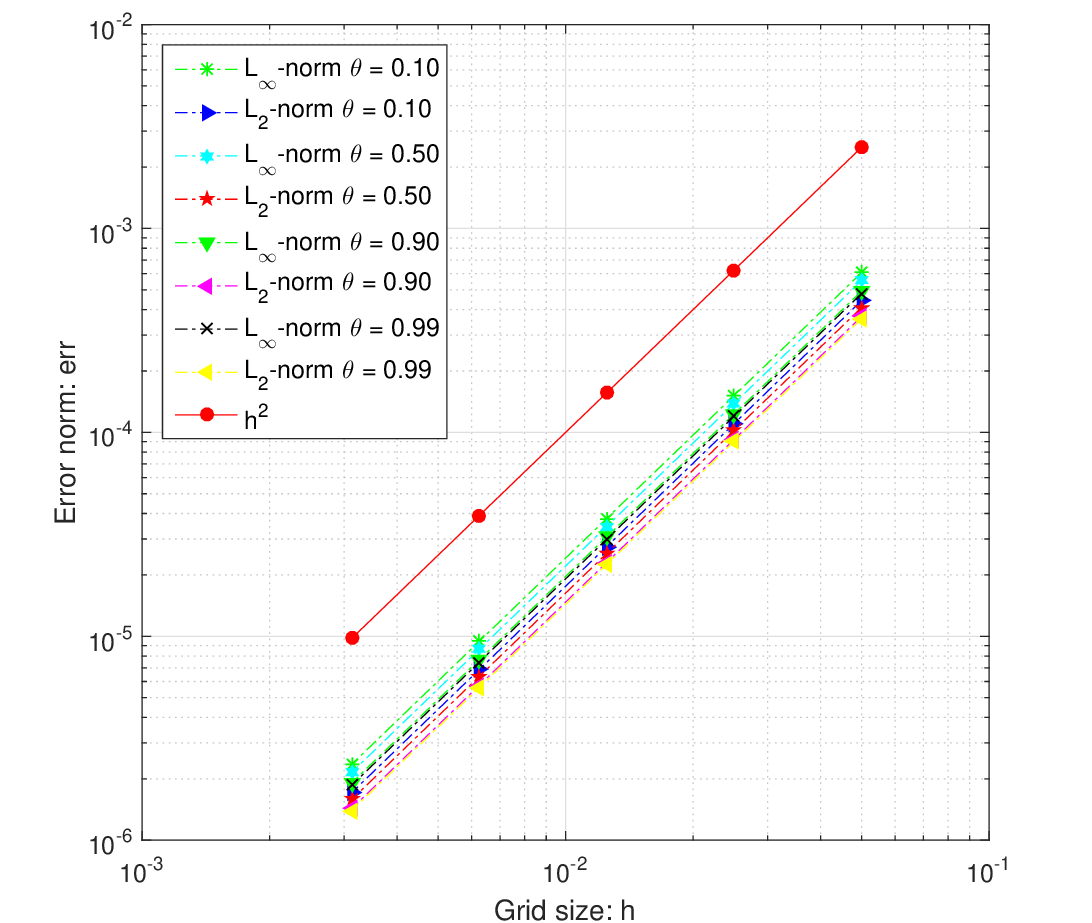}
\caption{Comparison the order of accuracy obtained by our proposed IDS for Example 1 in
space and time variables. Left: time direction; Right: space direction.}
\label{fig:2}
\end{figure}

\noindent{\textbf{Example 2.}} In this example, we consider the nonlinear equation
on space interval $[0, L] = [0, 1]$ and time interval $[0, T] = [0, 1]$ with
advection coefficients $d_+(t) = sech(t)$, $d_-(t) = 4 sech(-t)$,
diffusion coefficients $e_+(t) = (2 + \cos(t))^2$, $e_-(t) = 2 + \cos^2(t)$,
and nonlinear term
\begin{equation*}
\begin{split}
g(u(x ,t), x, t) = & \frac{\sin(u(x, t))}{4} + \left( \frac{\Gamma(3 + \theta)}{\Gamma(3)} t^2 +
2 t^{1 - \theta} E_{1, 2 - \theta}(2t) \right) x^2 (1 - x)^2 \\
& + \left( t^{2 + \theta} + \exp(2t) \right)
\Big\{ \frac{\Gamma(3)}{\Gamma(3 - \alpha)} [d_+ x^{2 - \alpha} + d_- (1 - x)^{2 - \alpha}] \\
& - \frac{2 \Gamma(4)}{\Gamma(4 - \alpha)} [d_+ x^{3 - \alpha} + d_- (1 - x)^{3 - \alpha}]
+ \frac{\Gamma(5)}{\Gamma(5 - \alpha)} [d_+ x^{4 - \alpha} + d_- (1 - x)^{4 - \alpha}] \Big\} \\
& - t^{2 + \theta} \Big\{ \frac{\Gamma(3)}{\Gamma(3 - \beta)} [e_+ x^{2 - \beta} + e_- (1 - x)^{2 - \beta}]
- \frac{2 \Gamma(4)}{\Gamma(4 - \beta)} [e_+ x^{3 - \beta} + e_- (1 - x)^{3 - \beta}] \\
& + \frac{\Gamma(5)}{\Gamma(5 - \beta)} [e_+ x^{4 - \beta} + e_- (1 - x)^{4 - \beta}] \Big\}
 - \frac{\sin\left( (t^{2 + \theta} + \exp(2t)) x^2 (1 - x)^2 \right)}{4},
\end{split}
\end{equation*}
in which $E_{\mu, \nu}(z)$ is the Mittag-Leffler function with two parameters defined by
\begin{equation*}
E_{\mu, \nu}(z) = \sum_{k = 0}^{\infty} \frac{z^{k}}{\Gamma(\mu k + \nu)}.
\end{equation*}

The exact solution is $u(x, t) = \left( t^{2 + \theta} + \exp(2t) \right) x^2 (1 - x)^2 $.
For the finite difference discretization, the space step and time step are taken
to be $h = 1/N$ and $\tau = h$, respectively.
The errors ($\xi = U - u$) and convergence order (CO)
in the norms $\left \| \cdot \right \|_0$ and
$\left \| \cdot \right \|_{\mathcal{C}(\bar{\omega}_{h \tau})}$,
where $\left \| U \right \|_{\mathcal{C}(\bar{\omega}_{h \tau})}
=$ max$_{(x_i,t_j) \in \bar{\omega}_{h \tau}} |U|$, are given in Example 1.
\vspace{4mm}

As can be seen in Tables \ref{tab5}-\ref{tab8}, the numerical solution provided by the difference approximation
is in good agreement with our theoretical analysis. In Tables \ref{tab5} and \ref{tab7}, we take $h = 1/800$,
the errors in maximum norm and $L_2$-norm decrease steadily with the shortening of time step,
and the convergence order of time is the expected $\mathcal{O}(\tau^2)$. While in Tables \ref{tab6} and \ref{tab8},
the mesh size $\tau = h$ is chosen and the spatial convergence rates of the scheme (\ref{eq2.10}) are also nearly two,
for $\theta = 0.1, 0.5, 0.9, 0.99$.
Furthermore, Figs. \ref{fig:3} and \ref{fig:4} are plotted to further illustrate the reliability of our proposed
scheme, i.e., the slopes of the error curves in this log-log plot are 2, for $\theta = 0.1, 0.5, 0.9, 0.99$.

\begin{table}[h]
\caption{$L_2$-norm and maximum norm error behavior versus
$\tau$-grid size reduction when $\alpha = 0.6$, $\beta = 1.8$ and $h = 1/800$ in Example 2.}
\centering
\begin{tabular}{cccccc}
\hline
$\theta$ & $\tau$ & $ \left \| \xi \right \|_{\mathcal{C}(\bar{\omega}_{h \tau})}$
& CO in $ \left \| \cdot \right \|_{\mathcal{C}(\bar{\omega}_{h \tau})} $
& max$_{0 \leq n \leq M} \left \| \xi^n \right \|_0$ & CO in $\left \| \cdot \right \|_0$ \\
\hline
0.10	&1/8	&6.7718e-04&--&	4.3108e-04	&-- \\
	&1/16	&1.7300e-04	&1.9688 &	1.1009e-04&	1.9693 \\
	&1/32&	4.1798e-05&	2.0493 &	2.6564e-05	&2.0511 \\
	&1/64&	8.3378e-06&	2.3257 &	5.2789e-06&	2.3312 \\
	&1/128	&1.6425e-06	&2.3438 	&4.9654e-07	&3.4103 \\
\hline
0.50&	1/8&	2.8747e-03&--&	1.8297e-03&	-- \\
	&1/16	&7.3181e-04&	1.9739 	&4.6572e-04	&1.9741 \\
	&1/32&	1.8254e-04	&2.0033 	&1.1613e-04	&2.0037 \\
	&1/64	&4.3648e-05	&2.0642 &	2.7733e-05	&2.0661 \\
	&1/128	&8.7404e-06	&2.3201 &	5.5326e-06&	2.3256 \\
\hline
0.90&	1/8&	4.0951e-03&	--&	2.6098e-03	&-- \\
	&1/16&	1.0327e-03&	1.9875 &	6.5724e-04	&1.9894 \\
	&1/32&	2.5667e-04&	2.0084 	&1.6330e-04&	2.0089 \\
	&1/64	&6.2038e-05	&2.0487 	&3.9436e-05&	2.0499 \\
	&1/128&	1.3330e-05&	2.2185 &	8.4480e-06&	2.2228 \\

\hline
0.99	&1/8&	4.1885e-03	&--&	2.6801e-03&	-- \\
&	1/16&	1.0594e-03&	1.9832 &	6.7470e-04&	1.9900 \\
	&1/32&	2.6328e-04	&2.0086 &	1.6752e-04	&2.0099 \\
	&1/64	&6.3658e-05	&2.0482 &	4.0472e-05&	2.0493 \\
	&1/128	&1.3737e-05&	2.2123 	&8.7084e-06	&2.2164 \\
\hline
\end{tabular}
\label{tab5}
\end{table}
\begin{table}[h]
\caption{$L_2$-norm and maximum norm error behavior versus
grid size reduction when $\alpha = 0.6$, $\beta = 1.8$ and $\tau = h$ in Example 2.}
\centering
\begin{tabular}{cccccc}
\hline
$\theta$ & $\tau$ & $ \left \| \xi \right \|_{\mathcal{C}(\bar{\omega}_{h \tau})}$
& CO in $ \left \| \cdot \right \|_{\mathcal{C}(\bar{\omega}_{h \tau})} $
& max$_{0 \leq n \leq M} \left \| \xi^n \right \|_0$ & CO in $\left \| \cdot \right \|_0$ \\
\hline
0.10 	&1/20	&5.2772e-03	&--&	3.6233e-03&-- \\
	&1/40	&1.2913e-03	&2.0309 &	8.7967e-04	&2.0423 \\
	&1/80	&3.1470e-04	&2.0368 &	2.1331e-04	&2.0440 \\
&	1/160	&7.6674e-05&	2.0372 &	5.1742e-05&	2.0435 \\
	&1/320	&1.8683e-05&	2.0370& 	1.2566e-05	&2.0418 \\
\hline
0.50 &	1/20&	4.8930e-03	&--&	3.3744e-03	&-- \\
	&1/40	&1.1943e-03	&2.0346 	&8.1793e-04	&2.0446 \\
&	1/80	&2.9056e-04	&2.0393 &	1.9797e-04&	2.0467 \\
	&1/160&	7.0647e-05	&2.0401 	&4.7934e-05	&2.0462 \\
&	1/320	&1.7181e-05	&2.0398 &	1.1621e-05	&2.0443 \\
\hline
0.90 &	1/20	&4.6652e-03&--&	3.2255e-03	&-- \\
	&1/40	&1.1378e-03	&2.0357 	&7.8171e-04	&2.0448 \\
	&1/80&	2.7670e-04&	2.0399& 	1.8908e-04	&2.0476 \\
	&1/160	&6.7216e-05	&2.0414 &	4.5748e-05	&2.0472 \\
	&1/320	&1.6330e-05	&2.0413 &	1.1083e-05	&2.0454 \\
\hline
0.99	&1/20&	4.6390e-03	&--&	3.2055e-03&-- \\
	&1/40	&1.1310e-03	&2.0362 	&7.7718e-04	&2.0442 \\
&	1/80	&2.7508e-04	&2.0397 &	1.8801e-04	&2.0474 \\
&	1/160&	6.6816e-05&	2.0416 &	4.5488e-05	&2.0473 \\
&	1/320	&1.6231e-05	&2.0414& 	1.1020e-05	&2.0454 \\
\hline
\end{tabular}
\label{tab6}
\end{table}
\begin{figure}[!htbp]
\centering
\includegraphics[width=3.12in,height=3.0in]{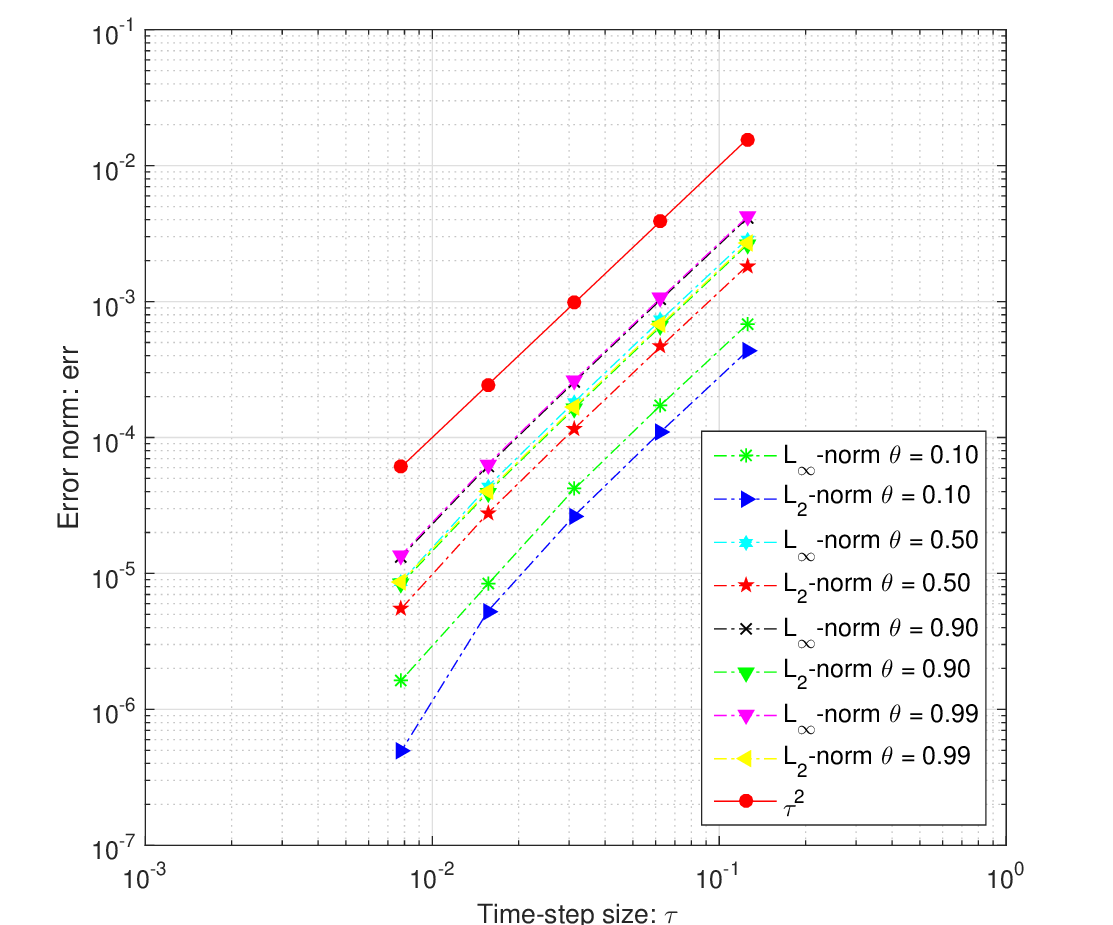}
\includegraphics[width=3.12in,height=3.0in]{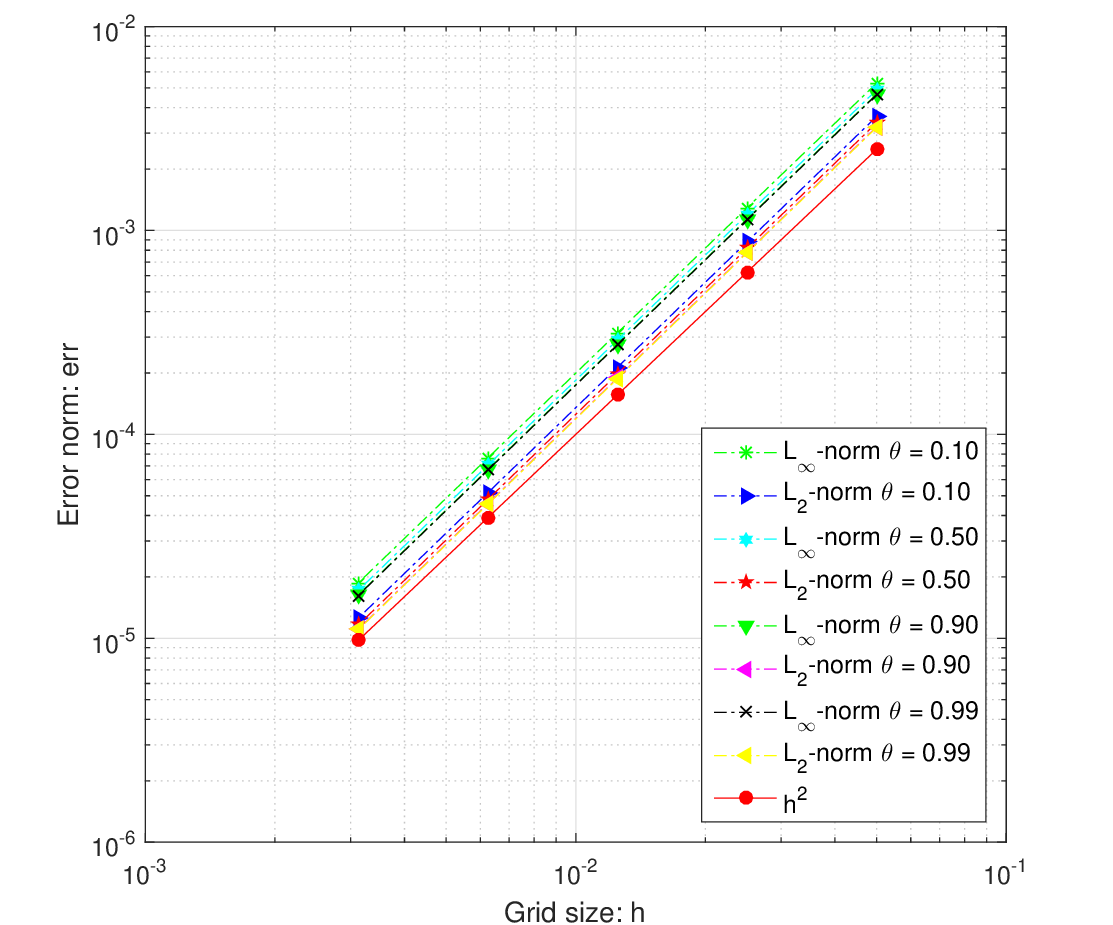}
\caption{Comparison the order of accuracy obtained by our proposed IDS for Example 2 in
space and time variables. Left: time direction; Right: space direction.}
\label{fig:3}
\end{figure}
\begin{table}[h]
\caption{$L_2$-norm and maximum norm error behavior versus
$\tau$-grid size reduction when $\alpha = 0.99$, $\beta = 1.99$ and $h = 1/800$ in Example 2.}
\centering
\begin{tabular}{cccccc}
\hline
$\theta$ & $\tau$ & $ \left \| \xi \right \|_{\mathcal{C}(\bar{\omega}_{h \tau})}$
& CO in $ \left \| \cdot \right \|_{\mathcal{C}(\bar{\omega}_{h \tau})} $
& max$_{0 \leq n \leq M} \left \| \xi^n \right \|_0$ & CO in $\left \| \cdot \right \|_0$ \\
\hline
0.10 &	1/8	&6.7919e-04&--&	4.3233e-04	&-- \\
	&1/16	&1.7333e-04	&1.9703& 	1.1015e-04	&1.9727 \\
	&1/32&	4.1678e-05&	2.0562& 	2.6300e-05	&2.0663 \\
	&1/64	&8.0974e-06	&2.3638& 	4.9278e-06	&2.4160 \\
	&1/128&	9.5365e-07&	3.0859& 	6.4747e-07	&2.9281 \\
\hline
0.50 &	1/8&	2.8857e-03	&--&	1.8374e-03	&-- \\
	&1/16&	7.3455e-04&	1.9740& 	4.6750e-04	&1.9746 \\
	&1/32&	1.8305e-04	&2.0046 &	1.1631e-04	&2.0070 \\
	&1/64&	4.3571e-05	&2.0708& 	2.7500e-05&	2.0805 \\
	&1/128&	8.5104e-06&	2.3561 &	5.1893e-06&	2.4058 \\
\hline
0.90 &	1/8&	4.1053e-03	&--&	2.6172e-03	&-- \\
	&1/16&	1.0367e-03&	1.9855 	&6.5998e-04&	1.9875 \\
	&1/32&	2.5752e-04&	2.0092 	&1.6373e-04&	2.0111 \\
	&1/64&	6.2046e-05&	2.0533 	&3.9263e-05&	2.0601 \\
	&1/128&	1.3122e-05&	2.2414 	&8.1204e-06	&2.2735 \\
\hline
0.99 &	1/8&	4.1843e-03	&--&	2.6789e-03	& -- \\
	&1/16&	1.0619e-03	&1.9783 	&6.7724e-04&	1.9839 \\
	&1/32&	2.6411e-04	&2.0074 	&1.6791e-04	&2.0120 \\
	&1/64&	6.3651e-05	&2.0529 	&4.0288e-05&	2.0593 \\
	&1/128	&1.3524e-05&	2.2347 	&8.3771e-06&	2.2658 \\
\hline
\end{tabular}
\label{tab7}
\end{table}
\begin{table}[h]
\caption{$L_2$-norm and maximum norm error behavior versus
grid size reduction when $\alpha = 0.99$, $\beta = 1.99$ and $\tau = h$ in Example 2.}
\centering
\begin{tabular}{cccccc}
\hline
$\theta$ & $\tau$ & $ \left \| \xi \right \|_{\mathcal{C}(\bar{\omega}_{h \tau})}$
& CO in $ \left \| \cdot \right \|_{\mathcal{C}(\bar{\omega}_{h \tau})} $
& max$_{0 \leq n \leq M} \left \| \xi^n \right \|_0$ & CO in $\left \| \cdot \right \|_0$ \\
\hline
0.10 &	1/20&	5.1023e-03	&--&	3.7233e-03	&-- \\
	&1/40&	1.2730e-03&	2.0029& 	9.2852e-04	&2.0036 \\
	&1/80&	3.1774e-04	&2.0023 &	2.3151e-04	&2.0039 \\
	&1/160	&7.9261e-05&	2.0032 &	5.7718e-05	&2.0040 \\
	&1/320	&1.9772e-05&	2.0032& 	1.4389e-05	&2.0041 \\
\hline
0.50 	&1/20&	4.7154e-03	&--&	3.4794e-03	&-- \\
	&1/40&	1.1767e-03	&2.0026 	&8.6769e-04	&2.0036 \\
	&1/80&	2.9361e-04	&2.0028 	&2.1629e-04	&2.0042 \\
&1/160	&7.3226e-05&	2.0035 	&5.3912e-05&	2.0043 \\
	&1/320&	1.8263e-05&	2.0034 	&1.3438e-05&	2.0043 \\
\hline
0.90 &	1/20&	4.4905e-03	&--&	3.3363e-03	&-- \\
	&1/40&	1.1217e-03	&2.0012 	&8.3257e-04&	2.0026 \\
	&1/80&	2.7992e-04	&2.0026 	&2.0758e-04&	2.0039 \\
	&1/160&	6.9818e-05&	2.0033 	&5.1743e-05&	2.0042 \\
	&1/320&	1.7413e-05&	2.0034 	&1.2897e-05	&2.0043 \\
\hline
0.99 &	1/20&	4.4668e-03	&--&	3.3193e-03	&-- \\
	&1/40&	1.1155e-03&	2.0016 	&8.2848e-04	&2.0023 \\
&	1/80&	2.7838e-04&	2.0026 	&2.0657e-04	&2.0038 \\
	&1/160&	6.9435e-05&	2.0033 	&5.1492e-05&	2.0042 \\
	&1/320&	1.7317e-05&	2.0035 	&1.2834e-05&	2.0044 \\
\hline
\end{tabular}
\label{tab8}
\end{table}
\begin{figure}[H]
\centering
\includegraphics[width=3.12in,height=3.0in]{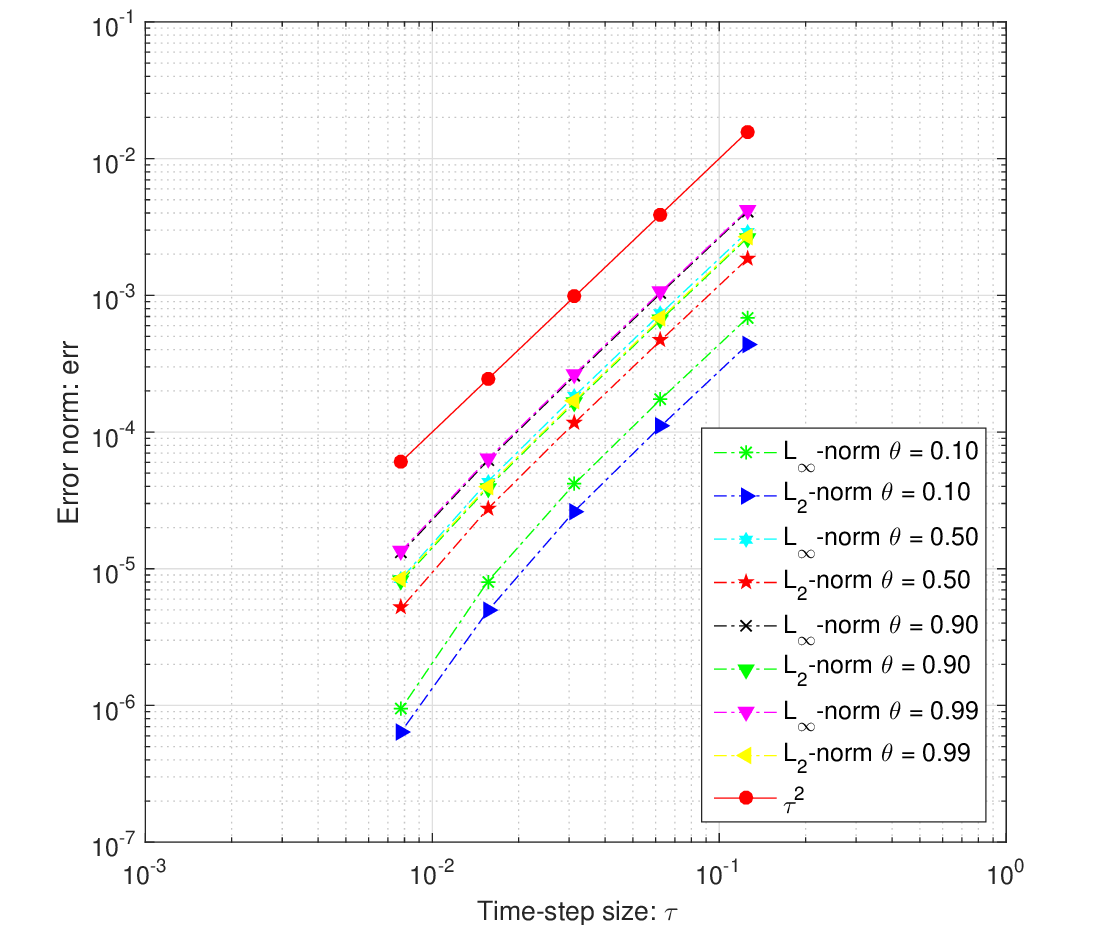}
\includegraphics[width=3.12in,height=3.0in]{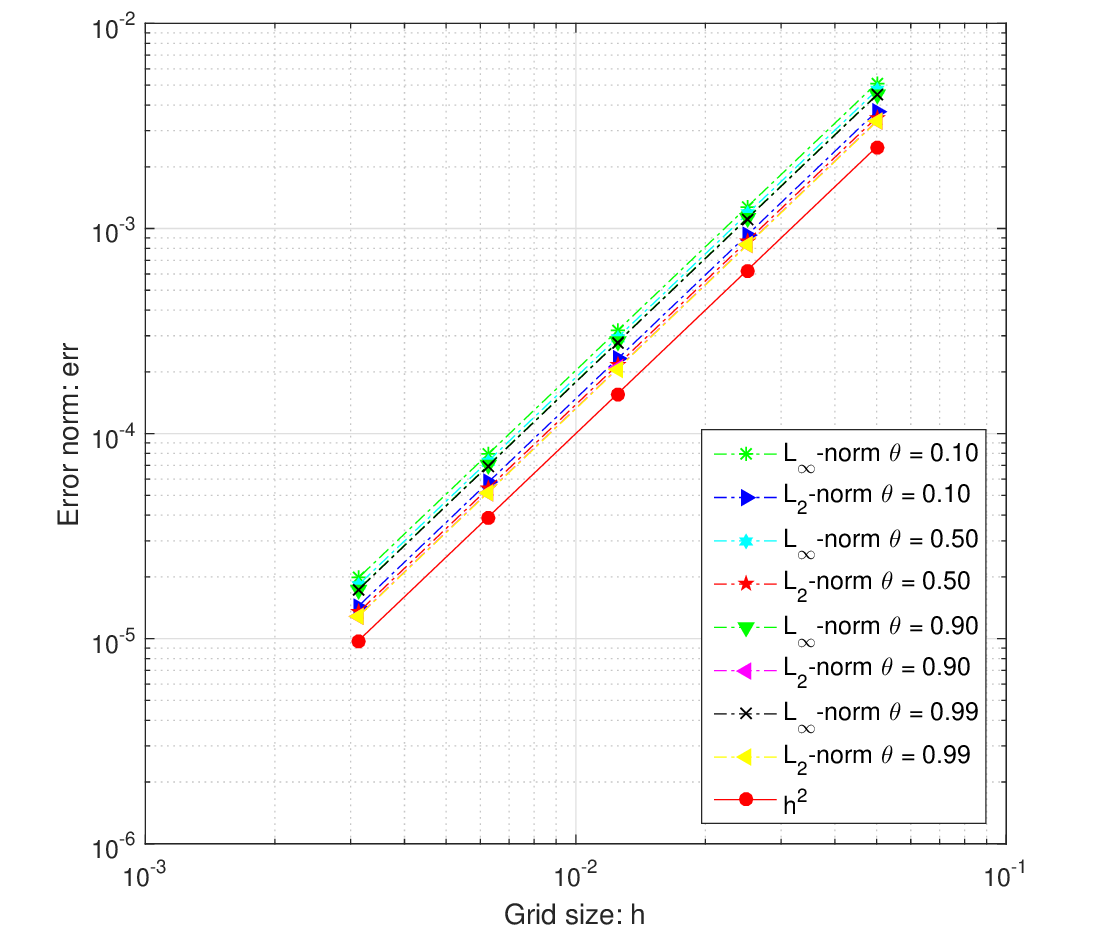}
\caption{Comparison the order of accuracy obtained by our proposed IDS for Example 2 in
space and time variables. Left: time direction; Right: space direction.}
\label{fig:4}
\end{figure}

\subsection{Fast implementations}
\label{sec4.2} \quad\

\noindent{\textbf{Example 3.}} In this example, we consider the case of linear situation (\ref{eq1.1}),
and the exact solution is given in Example 1.
In MATLAB code, we take the commands of generating matrices as the start point of \texttt{Time1},
and in the rest tables,
``\texttt{Speed-up}'' denotes ${\frac{\texttt{Time1}}{\texttt{Time}k}~(k = 2,3)}$
and ``\texttt{Iter}'' is the average number of iterations required for solving the TSADE problem,
namely,
\begin{equation*}
\texttt{Iter} = \frac{1}{M}
\sum\limits_{m = 1}^{M} \texttt{Iter}(m),
\end{equation*}
where $\texttt{Iter}(m)$ represents the number of iterations required for solving (\ref{eq3.1}).
As for GPBiCOR($m$, $\ell$) method \cite{gu2015bicorstab} with circulant preconditioners in this example,
a large number of experiments reveal that the results are better than others,
when $m = 3$ and $\ell = 1$. Some eigenvalue plots about both
original and preconditioned matrices are drawn in Figs. \ref{fig:5}-\ref{fig:6}.
These two figures confirm that for circulant preconditioning,
the eigenvalues of preconditioned matrices are clustered at 1 except for a few of them.
That is to say, the vast majority of the eigenvalues are well separated away from 0.
We validate the effectiveness and robustness of the designed circulant preconditioner from the
respective of clustering spectrum distribution.

The numerical results in Example 3 are presented in Tables \ref{tab9}-\ref{tab10},
reflecting that the average iterations of the two Krylov subspace methods
with circulant preconditioners have little difference.
In other words, the number of average iterations of these methods is about 6.6.
On the other hand, the CPU time by the two Krylov subspace methods
with circulant preconditioners are much less than
that by LU factorization method as $M$ and $N$ become large.
When $M = N = 2^{10}$ in Table \ref{tab9}, the CPU time is about 18.5 seconds,
and the \texttt{Speed-up} is almost 7.4, especially for PBiCGSTAB method,
the \texttt{Speed-up} can reach about 8.21. At the same case,
in Table \ref{tab10} the CPU time is about 18.7 seconds,
and the \texttt{Speed-up} is almost 7.7.
Meanwhile, although \texttt{Time1} required by LU factorization method
for small test problems ($M = N = 2^{6},2^{7}$) is the cheapest
among other two methods (i.e. PBiCGSTAB and PGPBiCOR($m$, $\ell$)),
our proposed methods are still more attractive in aspects of the lower memory requirement.
\begin{table}[H]\small\tabcolsep=0.4cm
\begin{center}
\caption{{\small {CPU time in seconds for solving Example 1 with $\theta = 0.9$,
$\alpha = 0.8$, $\beta = 1.9$.}}}
\begin{tabular}{cccccc}
\hline  & & \multicolumn{2}{c}{$\rm{PBiCGSTAB}$} & \multicolumn{2}{c}{$\rm{PGPBiCOR(3,1)}$} \\
[-2pt] \cmidrule(lr){3-4} \cmidrule(lr){5-6} \\ [-11pt]
$h = \tau$ & \texttt{Time1} & \texttt{Time2(Iter)} & $\texttt{Speed-up}$
& \texttt{Time3(Iter)} & $\texttt{Speed-up}$  \\
\hline
$2^{-6}$ & 0.031 & 0.161(6.9) & 0.19  & 0.213(5.6) & 0.14  \\
$2^{-7}$ & 0.145 & 0.488(7.0) & 0.30 & 0.612(6.1) & 0.24  \\
$2^{-8}$ & 1.624 & 1.217(7.0) & 1.33 &1.500(6.3) & 1.08  \\
$2^{-9}$ &17.375 & 4.828(7.0) & 3.60 & 5.592(6.5) & 3.11  \\
$2^{-10}$ & 143.968 & 17.544(7.0) & 8.21 &19.393(6.7) & 7.42 \\
\hline
\end{tabular}
\label{tab9}
\end{center}
\end{table}
\begin{figure}[H]
\centering
\includegraphics[width=3.12in,height=3.0in]{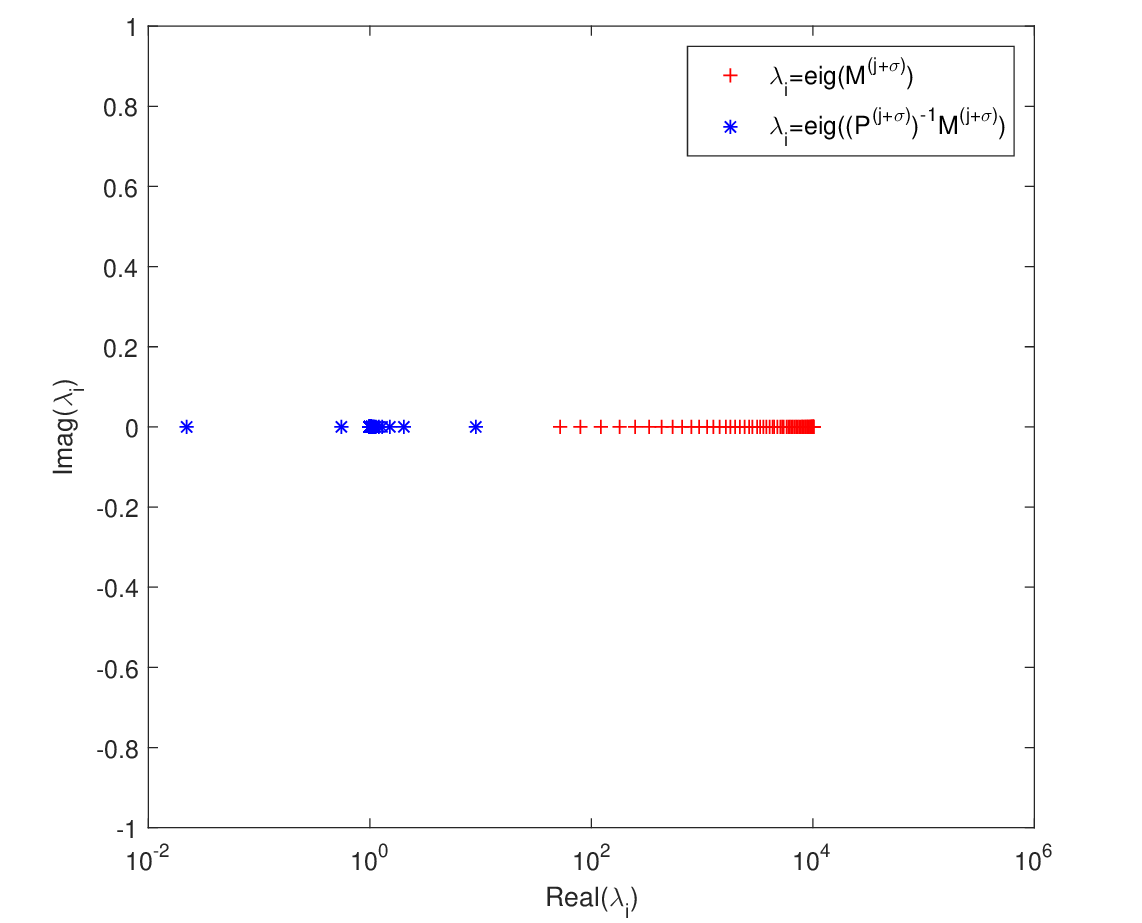}
\includegraphics[width=3.12in,height=3.0in]{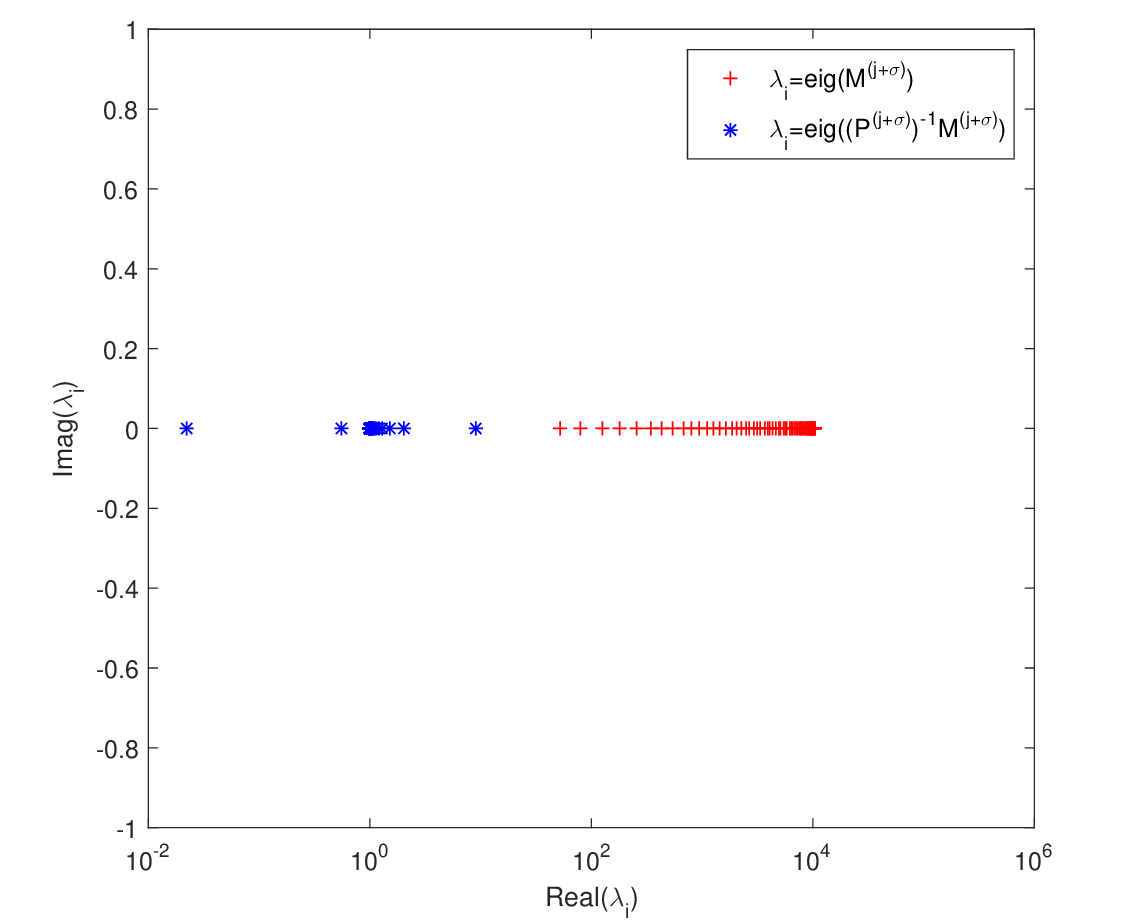}
\caption{Spectrum of both original and preconditioned matrices,
when $M = N = 2^6$, $\theta = 0.9$,
$\alpha = 0.8$, $\beta = 1.9$.
Left: Time level $j = 0$; Right: Time level $j = 1$.}
\label{fig:5}
\end{figure}
\begin{table}[H]\small\tabcolsep=0.4cm
\begin{center}
\caption{{\small {CPU time in seconds for solving Example 1 with $\theta = 0.7$,
$\alpha = 0.6$, $\beta = 1.5$.}}}
\begin{tabular}{cccccc}
\hline  & & \multicolumn{2}{c}{$\rm{PBiCGSTAB}$}& \multicolumn{2}{c}{$\rm{PGPBiCOR(3,1)}$} \\
[-2pt] \cmidrule(lr){3-4} \cmidrule(lr){5-6} \\ [-11pt]
$h = \tau$ & \texttt{Time1} & \texttt{Time2(Iter)} & $\texttt{Speed-up}$
& \texttt{Time3(Iter)} & $\texttt{Speed-up}$ \\
\hline
$2^{-6}$ & 0.031 & 0.167(6.8) & 0.18 & 0.154(6.0) & 0.20  \\
$2^{-7}$ & 0.149 & 0.509(7.0) & 0.29 &0.485(6.3)  & 0.31  \\
$2^{-8}$ & 1.701 & 1.335(7.8) & 1.27 &1.281(7.2) &  1.33 \\
$2^{-9}$ & 17.251 & 5.325(8.0) & 3.24 &5.275(7.6) &  3.27 \\
$2^{-10}$ & 143.739 & 19.166(8.1) & 7.50 &18.261(7.8) &  7.87 \\
\hline
\end{tabular}
\label{tab10}
\end{center}
\end{table}
\begin{figure}[H]
\centering
\includegraphics[width=3.12in,height=3.0in]{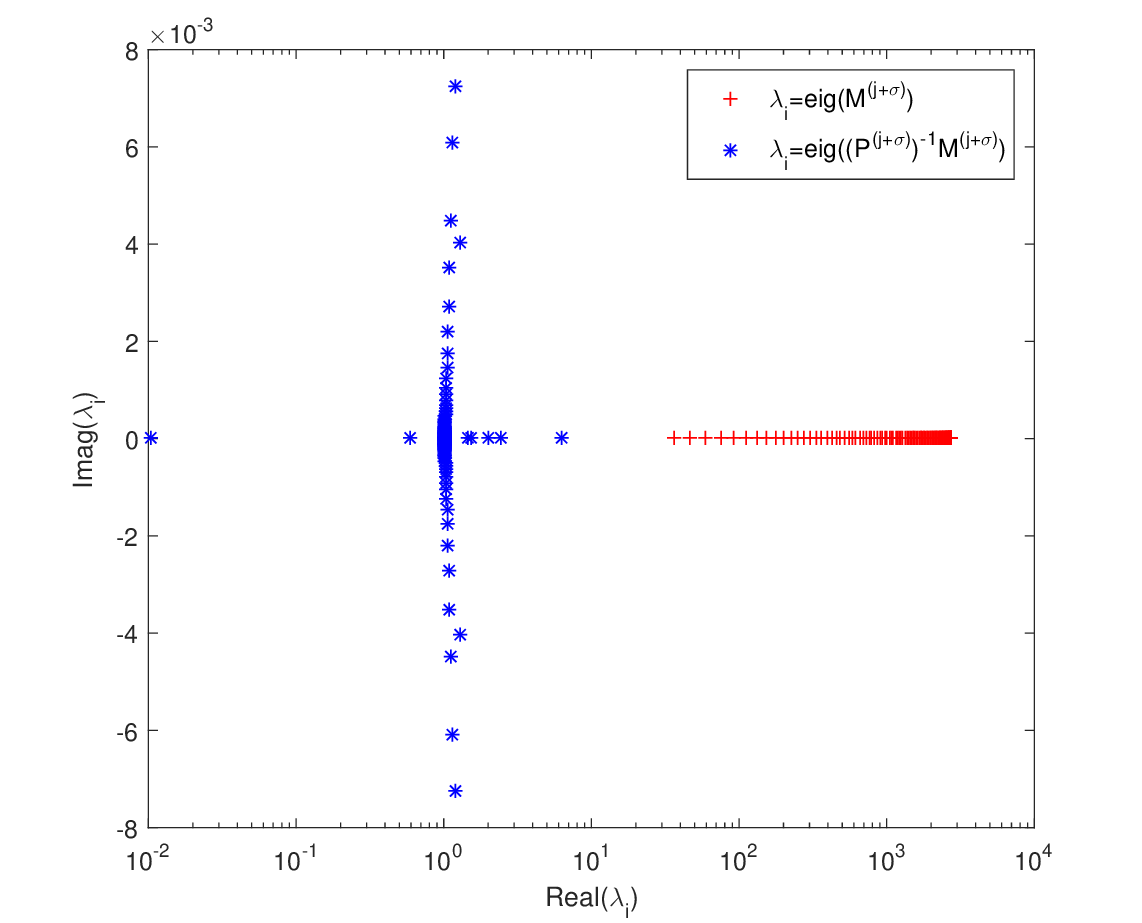}
\includegraphics[width=3.12in,height=3.0in]{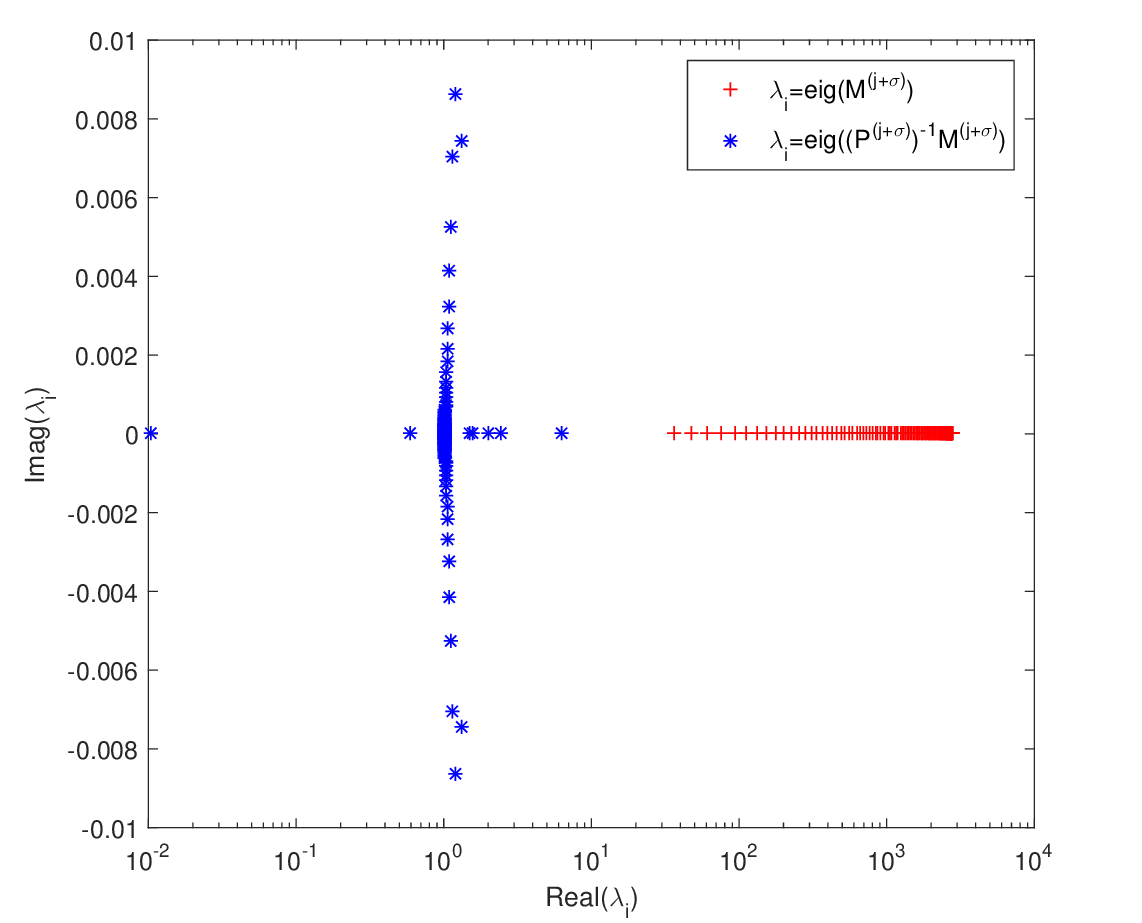}
\caption{Spectrum of both original and preconditioned matrices,
when $M = N = 2^7$, $\theta = 0.7$,
$\alpha = 0.6$, $\beta = 1.5$.
Left: Time level $j = 0$; Right: Time level $j = 1$.}
\label{fig:6}
\end{figure}

\noindent{\textbf{Example 4.}} In this example, the exact solution and nonlinear source term are given in Example 2.
For performing algorithm (\ref{eq3.4})-(\ref{eq3.6}), the outer iteration tolerance is chosen as $10^{-12}$.
As for GPBiCOR($m$, $\ell$) method \cite{gu2015bicorstab} with circulant preconditioners in this example,
a large number of experiments reveal that the results are better than others,
when $m = 3$ and $\ell = 1$. Some eigenvalue plots about both
original and preconditioned matrices are drawn in Figs. \ref{fig:7}-\ref{fig:8}.
These two figures confirm that for circulant preconditioning,
the eigenvalues of preconditioned matrices are clustered at 1 except for a few of them.
That is to say, the vast majority of the eigenvalues are well separated away from 0.
We validate the effectiveness and robustness of the designed circulant preconditioner from the
respective of clustering spectrum distribution.

In Tables \ref{tab11}-\ref{tab12}, ``\texttt{Iter1}" represents the average number of iterations
required for solving the nonlinear problem with algorithm (\ref{eq3.4})-(\ref{eq3.6}), i.e.,
\begin{equation*}
\texttt{Iter1} = \frac{1}{M}\sum\limits_{m = 1}^{M} \texttt{Iter1}(m),
\end{equation*}
in which $\texttt{Iter1}(m)$ is the number of outer iteration required to solve (\ref{eq3.4}),
and ``\texttt{Iter2}" denotes the average iterative number
of a preconditioned Krylov subspace method required to solve (\ref{eq3.4}), namely,
\begin{equation*}
 \texttt{Iter2} =  \frac{1}{M}\sum\limits_{m = 1}^{M}
 \frac{1}{\texttt{Iter1}(m)} \sum\limits_{k = 1}^{\texttt{Iter1}(m)} \texttt{Iter2}(k)
\end{equation*}
in which $\texttt{Iter2}(k)$ is the iterative number of the preconditioned Krylov subspace method
required at $k$th outer iteration.
In view of the overall, the CPU time by the two Krylov subspace methods
with circulant preconditioners are much less than
that by LU factorization method as $M$ and $N$ become large.
When $M = N = 2^{10}$ in Table \ref{tab11}, the CPU time is about 41.8 seconds,
and the \texttt{Speed-up} is almost 3.8. At the same case,
Table \ref{tab12} indicates that the CPU time is about 52.3 seconds,
and the \texttt{Speed-up} is almost 3.2.
Meanwhile, although \texttt{Time1} required by LU factorization method
for small test problems ($M = N = 2^{6},2^{7},2^{8}$) is the cheapest
among other two methods (i.e. PBiCGSTAB and PGPBiCOR($m$, $\ell$)),
our proposed methods are still more attractive in aspects of the lower memory requirement.
As a conclusion, our proposed IDS with fast solution techniques is still more practical than LU factorization method.
\begin{table}[h]\small\tabcolsep=5.3pt
\begin{center}
\caption{{\small {CPU time in seconds for solving Example 2 with $\theta = 0.9$,
$\alpha = 0.8$, $\beta = 1.9$.}}}
\begin{tabular}{cccccc}
\hline  & & \multicolumn{2}{c}{$\rm{PBiCGSTAB}$} & \multicolumn{2}{c}{$\rm{PGPBiCOR(3,1)}$} \\
[-2pt] \cmidrule(lr){3-4} \cmidrule(lr){5-6}\\ [-11pt]
$h = \tau$ & \texttt{Time1(Iter1)} & \texttt{Time2(Iter1, Iter2)} & $\texttt{Speed-up}$
& \texttt{Time3(Iter1,Iter2)} & $\texttt{Speed-up}$ \\
\hline
$2^{-6}$&0.038(5.0) &0.470(5.0, 7.0)&0.08&0.482(5.0, 6.0)&0.08 \\
$2^{-7}$&0.164(4.1)&1.414(4.1, 7.0)&0.12&1.420(4.1, 6.0)&0.12 \\
$2^{-8}$&1.922(4.0)&3.325(4.0, 7.0)&0.58&3.350(4.0, 6.0)&0.57 \\
$2^{-9}$&19.122(4.0)&13.923(4.0, 7.0)&1.37&13.879(4.0, 6.0)&1.38 \\
$2^{-10}$&159.561(3.5)&42.497(3.5, 7.4)&3.75&41.005(3.5, 6.0)&3.89\\
\hline
\end{tabular}
\label{tab11}
\end{center}
\end{table}
\begin{figure}[H]
\centering
\includegraphics[width=3.12in,height=3.0in]{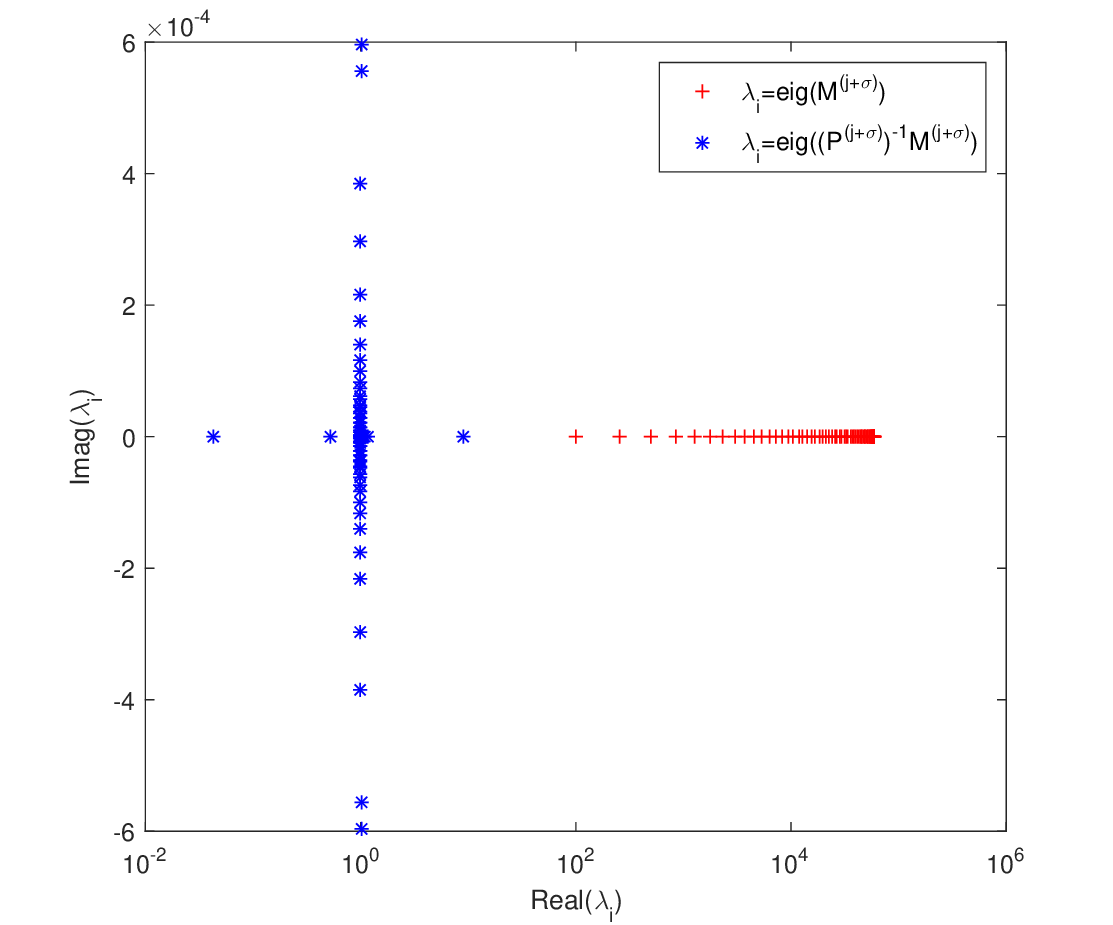}
\includegraphics[width=3.12in,height=3.0in]{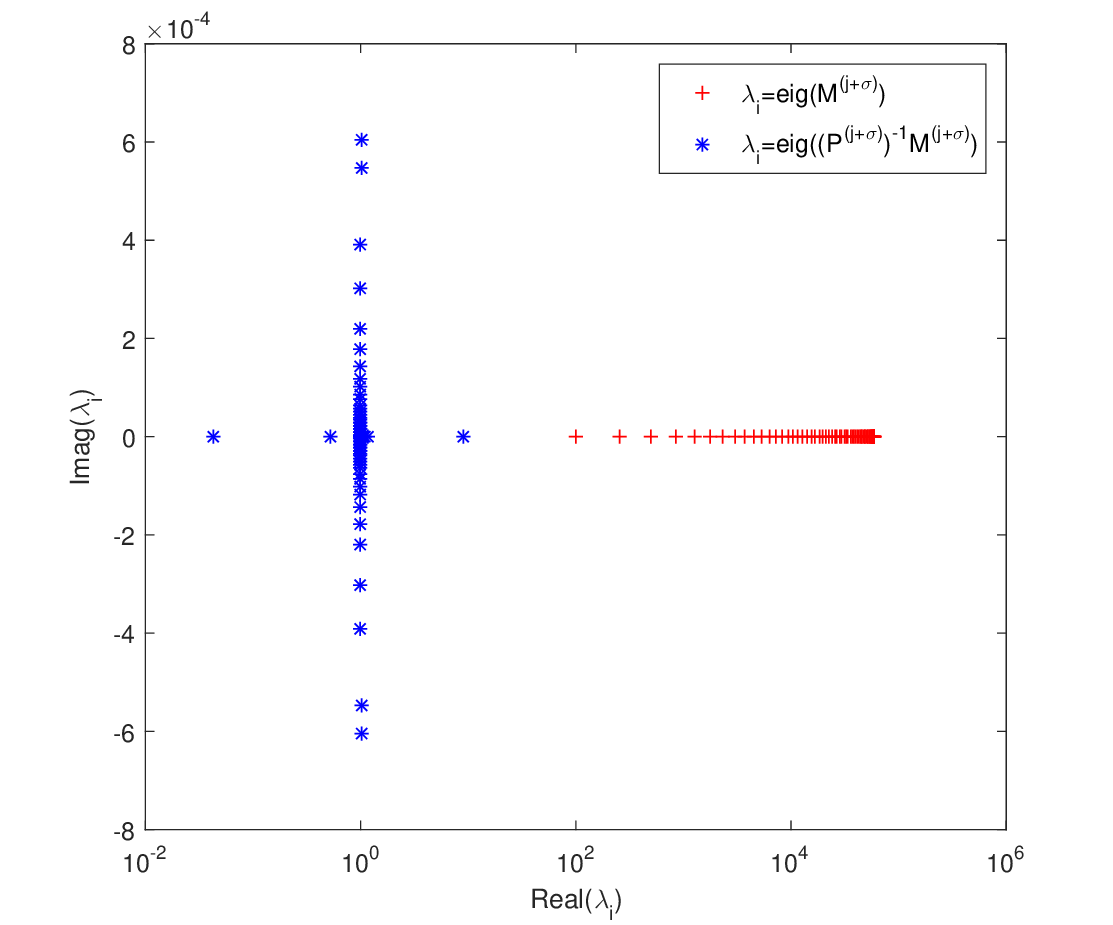}
\caption{Spectrum of both original and preconditioned matrices,
when $M = N = 2^6$, $\theta = 0.9$,
$\alpha = 0.8$, $\beta = 1.9$.
Left: Time level $j = 0$; Right: Time level $j = 1$.}
\label{fig:7}
\end{figure}
\begin{table}[h]\small\tabcolsep=5.3pt
\begin{center}
\caption{{\small {CPU time in seconds for solving Example 2 with $\theta = 0.7$,
$\alpha = 0.6$, $\beta = 1.7$.}}}
\begin{tabular}{cccccc}
\hline  & & \multicolumn{2}{c}{$\rm{PBiCGSTAB}$} & \multicolumn{2}{c}{$\rm{PGPBiCOR(3,1)}$} \\
[-2pt] \cmidrule(lr){3-4} \cmidrule(lr){5-6}\\ [-11pt]
$h = \tau$ & \texttt{Time1(Iter1)} & \texttt{Time2(Iter1, Iter2)} & $\texttt{Speed-up}$
& \texttt{Time3(Iter1, Iter2)} & $\texttt{Speed-up}$  \\
\hline
$2^{-6}$&0.037(5.0)&0.480(5.0, 7.0)&0.08&0.476(5.0, 6.0)&0.08 \\
$2^{-7}$&0.171(5.0)&1.666(5.0, 7.0)&0.10&1.866(5.0, 7.0)&0.09 \\
$2^{-8}$&1.802(4.3)&3.975(4.3, 8.1)&0.45&4.022(4.3, 7.0)&0.45 \\
$2^{-9}$&18.950(4.0)&15.604(4.0, 8.0)&1.21&15.873(4.0, 7.2)&1.19 \\
$2^{-10}$&164.387(4.0)&50.313(4.0, 8.0)&3.27&54.279(4.0, 7.8)&3.03 \\
\hline
\end{tabular}
\label{tab12}
\end{center}
\end{table}
\begin{figure}[H]
\centering
\includegraphics[width=3.12in,height=3.0in]{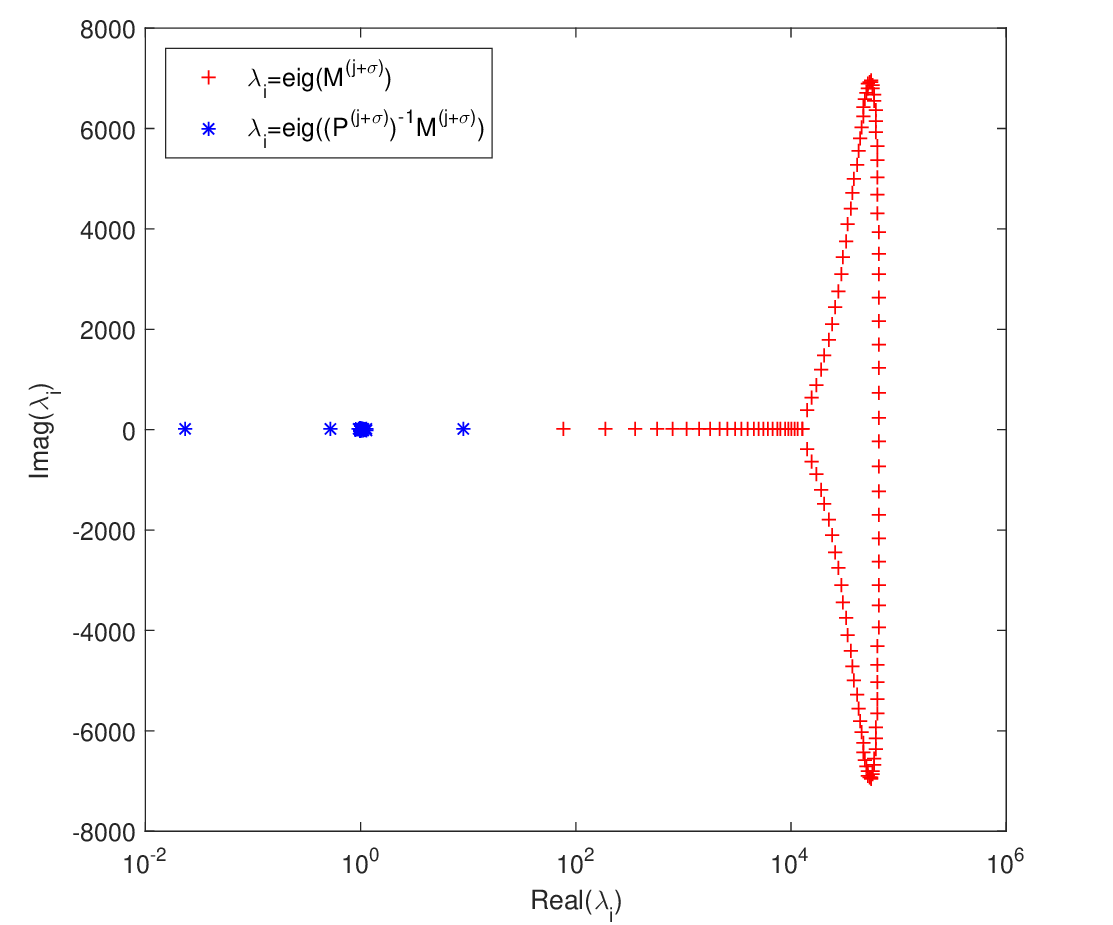}
\includegraphics[width=3.12in,height=3.0in]{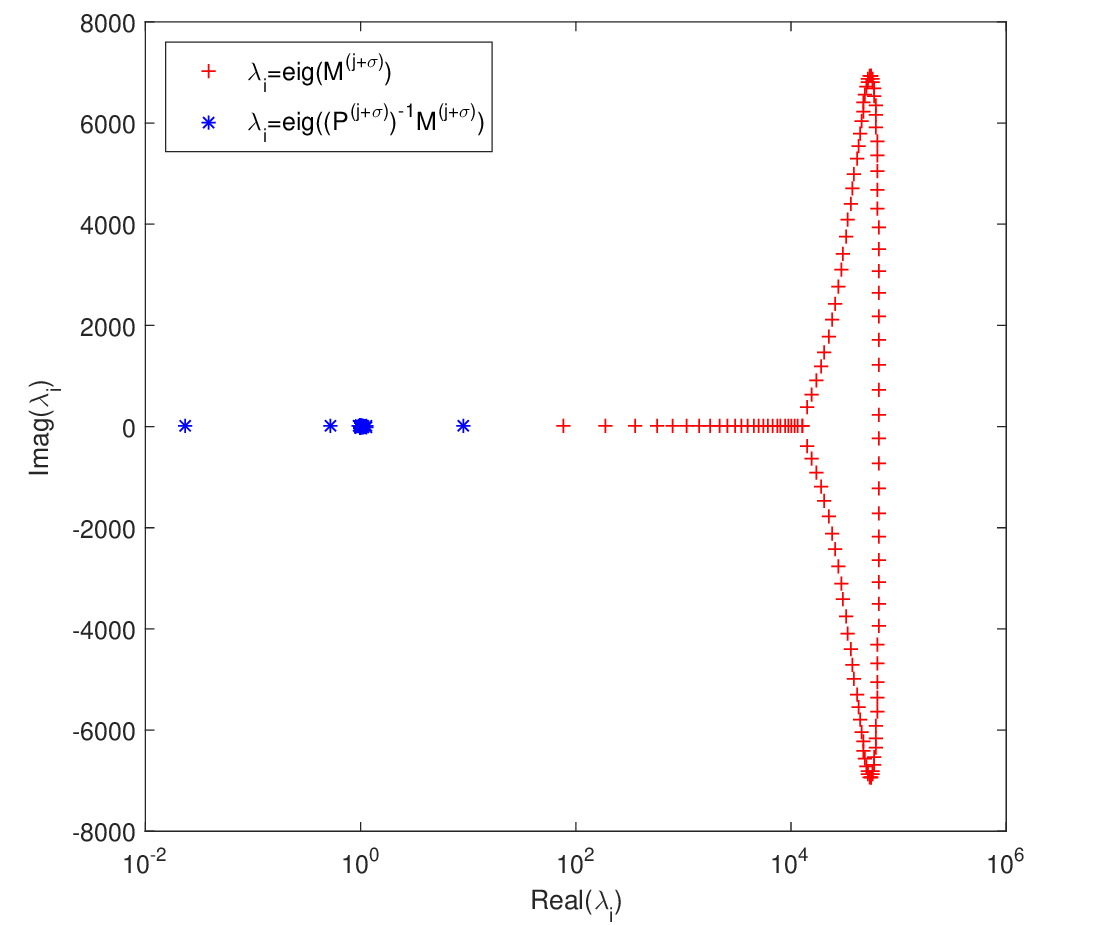}
\caption{Spectrum of both original and preconditioned matrices,
when $M = N = 2^7$, $\theta = 0.7$,
$\alpha = 0.6$, $\beta = 1.7$.
Left: Time level $j = 0$; Right: Time level $j = 1$.}
\label{fig:8}
\end{figure}

\section{Conclusion}
\label{sec5} \quad \
In this work, we constructed a difference scheme of the second order approximation
in space and time for the linear TSADE with variable coefficients.
Then, we utilized the same numerical approximation technique to
simulate the nonlinear problem of TSADE and the convergence order still was $\mathcal{O}(\tau^2 + h^2)$.
Numerical tests shown that our implicit difference schemes
converge with second-order accuracy in both time and space.
More significantly, we employed two reliable preconditioning iterative techniques,
which have only computational cost and memory of
$\mathcal{O}(N \log N)$ and $\mathcal{O}(N)$ respectively,
to ameliorate the calculation skill.
And numerical experiments strongly support the
efficiency of the two preconditioning methods.
In our future work, we will apply our IDM into
more linear/nonlinear fractional partial differential equations
and make a rigorous theoretical analysis of the stability
for the developed difference schemes.

\section*{Acknowledgments}
\addcontentsline{toc}{section}{Acknowledgments}
\label{sec6} \quad \
\textit{This work is supported by 973 Program (2013CB329404),
NSFC (1370147, 11501085 and 61402082),
the Fundamental Research Funds for the Central Universities (ZYGX2016J132 and ZYGX2016J138),
the Scientific Research Fund of Sichuan Provincial Education Department (15ZA0288).}
\section*{Appendix}
\addcontentsline{toc}{section}{Appendix}
\label{sec7} \quad \
\begin{algorithm}[H]
\caption{GPBiCOR(3,1) for $A x = b$ with preconditioner $K$.}
\label{alg1}
\begin{multicols}{2}
\begin{algorithmic}[1]
\small
\STATE {Select initial guess $\textbf{x}_0$ and $\textbf{r}_0 = \textbf{b} - A \textbf{x}_0$}
\STATE {Choose $\textbf{r}_0^{\ast} = A K^{-1} \textbf{r}_0$ s.t. $(\textbf{r}_0^{\ast},A K^{-1} \textbf{r}_0) \neq 0$. Set $\textbf{t}_{-1} = \textbf{w}_{-1} = \textbf{u}_{-1} = \hat{\textbf{u}}_{-1} = \textbf{0}$, $\beta_{-1} = 0$, $\textbf{e}_{0} = K^{-1} \textbf{r}_0$, $\hat{\textbf{e}}_{0} = A \textbf{e}_0$.}
\STATE {\textbf{for} $n = 0, 1, \cdots$, until convergence \textbf{do}}
\STATE\quad  {$\textbf{p}_n = \textbf{e}_n + \beta_{n-1}(\textbf{p}_{n-1} - \textbf{u}_{n-1})$}
\STATE\quad {$\hat{\textbf{p}}_n = \hat{\textbf{e}}_n + \beta_{n-1}(\hat{\textbf{p}}_{n-1} - \hat{\textbf{u}}_{n-1})$}
\STATE\quad {Solve $K \textbf{h}_n = \hat{\textbf{p}}_n$}
\STATE\quad {Compute $\textbf{g}_n = A \textbf{h}_n$ and then $\alpha_n = \frac{(\textbf{r}_0^{\ast},\textbf{e}_n)}
{(\textbf{r}_0^{\ast},\textbf{g}_n)}$}
\STATE\quad {$\textbf{t}_n = \textbf{r}_n - \alpha_n \hat{\textbf{p}}_n$}
\STATE\quad {$\textbf{y}_n = \textbf{t}_{n-1} - \textbf{t}_n - \alpha_n \textbf{q}_{n-1}$}
\STATE\quad {$\textbf{s}_n = \textbf{e}_n - \alpha_n \textbf{h}_n$ \qquad ($\triangleright \textbf{s}_n \triangleq K^{-1} \textbf{t}_n$)}
\STATE\quad {$\textbf{w}_n = \hat{\textbf{e}}_n - \alpha_n \textbf{g}_n$ \qquad($\triangleright \textbf{w}_n \triangleq A \textbf{s}_n$)}
\STATE\quad {\textbf{if} ($\mod(n,4) < 3$ or $n=0$) \textbf{then}}
\STATE\quad\quad {$\xi = \frac{(\textbf{w}_n,\textbf{t}_n)}{(\textbf{w}_n,\textbf{w}_n)}$
~(Hint: $\eta_n = 0$)}
\STATE\quad\quad {$\textbf{u}_n = \xi_n \textbf{h}_n$, $\hat{\textbf{u}}_n = \xi_n \textbf{g}_n$}
\STATE\quad\quad {$\textbf{z}_n = \xi_n \textbf{e}_n - \alpha_n \textbf{u}_n$}
\STATE\quad\quad {$\textbf{r}_{n+1} = \textbf{t}_n - \xi_n \textbf{w}_n$}
\STATE\quad {\textbf{else}}
\STATE\quad\quad {$\xi_n =
\frac{(\textbf{y}_n,\textbf{y}_n)(\textbf{w}_n,\textbf{t}_n) - (\textbf{y}_n,\textbf{t}_n)(\textbf{w}_n,\textbf{y}_n)}
{(\textbf{w}_n,\textbf{w}_n)(\textbf{y}_n,\textbf{y}_n)
- (\textbf{y}_n,\textbf{w}_n)(\textbf{w}_n,\textbf{y}_n)}$}
\STATE\quad\quad {$\eta_n =
\frac{(\textbf{w}_n,\textbf{w}_n)(\textbf{y}_n,\textbf{t}_n) - (\textbf{y}_n,\textbf{w}_n)(\textbf{w}_n,\textbf{t}_n)}
{(\textbf{w}_n,\textbf{w}_n)(\textbf{y}_n,\textbf{y}_n)
- (\textbf{y}_n,\textbf{w}_n)(\textbf{w}_n,\textbf{y}_n)}$}
\STATE\quad\quad {$\textbf{u}_n = \xi_n \textbf{h}_n + \eta_n(\textbf{s}_{n-1} - \textbf{e}_n + \beta_{n-1} \textbf{u}_{n-1})$}
\STATE\quad\quad {$\hat{\textbf{u}}_n = \xi_n \textbf{g}_n + \eta_n(\textbf{w}_{n-1} - \hat{\textbf{e}}_n + \beta_{n-1} \hat{\textbf{u}}_{n-1})$}
\STATE\quad\quad {$\textbf{z}_n = \xi_n \textbf{e}_n + \eta_n \textbf{z}_{n-1} - \alpha_n \textbf{u}_n$}
\STATE\quad\quad {$\textbf{r}_{n+1} = \textbf{t}_n - \eta_n \textbf{y}_n - \xi_n \textbf{w}_n$}
\STATE\quad {\textbf{endif}}
\STATE\quad {Solve $K \textbf{e}_{n+1} = \textbf{r}_{n+1}$}
\STATE\quad {Compute $\hat{\textbf{e}}_{n+1} = A \textbf{e}_{n+1}$}
\STATE\quad {$\beta_n = \frac{\alpha_n}{\xi_n} \cdot \frac{(\textbf{r}_0^{\ast}, \hat{\textbf{e}}_{n+1})}
{(\textbf{r}_0^{\ast}, \hat{\textbf{e}}_{n})}$}
\STATE\quad {$\textbf{x}_{n+1} = \textbf{x}_n + \alpha_n \textbf{p}_n + \textbf{z}_n$}
\STATE\quad {$\textbf{q}_n = \textbf{w}_n + \beta_n \hat{\textbf{p}}_n$}
\STATE {\textbf{end for}}
\end{algorithmic}
\end{multicols}
\end{algorithm}
\bibliography{reference}
\end{document}